\documentclass[11pt]{article}
\usepackage{amsmath, amssymb, amsfonts, amsthm, authblk, color, graphicx, enumitem, mathrsfs, indentfirst}
\usepackage{hyperref, cleveref}
\hypersetup{colorlinks=true, linkcolor=blue, filecolor=magenta, urlcolor=cyan}
\usepackage[textwidth=6in,textheight=9in]{geometry}
\allowdisplaybreaks

\newtheorem{theorem}{Theorem}
\newtheorem{lemma}{Lemma}[section]

\newtheorem{remark}{Remark}[section]

\newtheorem{problem}{Problem}

\newtheorem{example}{Example}

\numberwithin{equation}{section}

\newcommand{\keywords}[1]{\small\textbf{\textit{Keywords---}}#1}

\title{A data-driven model reduction approach for backward fractional diffusion-wave equations}

\author[1]{Dakang Cen}
\author[2]{Zhiyuan Li\footnote{Corresponding author 1: lizhiyuan@nbu.edu.cn, supported by the National Natural Science Foundation of China (no. 12271277), Ningbo Youth Leading Talent Project (no. 2024QL045).  and the Open Research Fund of the Key Laboratory of Nonlinear Analysis \& Applications (Central China Normal University), Ministry of Education, China (no. NAA20230RG002).}}

\author[3]{Wenlong Zhang\footnote{Corresponding author 2: zhangwl@sustech.edu.cn, supported by the National Natural Science Foundation of China under grant
numbers No.12371423 and No.12241104.}}
\affil[1,3]{Department of Mathematics, Southern University of Science and Technology, Shenzhen, 518055, China}
\affil[2]{School of Mathematics and Statistics, Ningbo University, Ningbo, 315211, China}

\begin{document}
\maketitle

\abstract{In this work, we propose an observation system based on the available data which solution is one-be-one mapping to the forward problem(with the unknown initial function) solution. It implies their solutions share the same linear structure in the finite dimensional space. Theoretical results show model reduction approaches constructed for the observation system also work well for the forward problem, which significantly improve the efficiency of solving the inverse problem.  Several numerical examples are presented to support our finding.}

\keywords{backward fractional wave, model reduction method, inverse crime}

\textbf{MSC2020:} 35R11, 35R09, 35B40

\section{Introduction}
Assuming that $\alpha\in(1,2)$ and that $\Omega \subset \mathbb{R}^d$, $d=1,2$, is a bounded domain with sufficiently smooth boundary $\partial \Omega$, we consider the following fractional wave equation:
\begin{equation}\label{eq-gov}
\begin{cases}
  \partial_t^\alpha u - \Delta u = f(x,t), & (x,t) \in \Omega \times (0,T), \\
  u(x,0) = a_0(x), & x \in \Omega, \\
  \frac{\partial}{\partial t}u(x,0) = a_1(x), & x \in \Omega, \\
  u(x,t) = 0, & (x,t) \in \partial \Omega \times (0,T),
\end{cases}
\end{equation}
where the operator $\partial_t^\alpha$ is referred to as the Caputo derivative of order $\alpha$, defined by
$$
\partial_t^\alpha \psi(t) := \frac{1}{\Gamma(2 - \alpha)} \int_0^t (t-\tau)^{1-\alpha} \psi''(\tau) d\tau, \quad t>0.
$$

In our previous work \cite{LiZhang2025,LiZhang2024}, we investigated recovering $a_1(x)$ from the terminal measurement of (\ref{eq-gov}). We established stability of the inversion and introduced a scattered point measurement-based regularization method. The optimal error estimates not only balance discretization error, the noise, and the number of observation points, but also propose an a priori choice of regularization parameters. The iteration optimization method is used to find the true initial value of the backward problem. It has to solve the forward diffusion-wave equation one or two times in each iteration. However, (\ref{eq-gov}) is a non-local model, the current solution $u(t_k)$, $k\in\mathbb{N}^+$ depends on all previous solutions $u(t_i)$, $i=0,\cdots,k-1$. The computation time is much longer than that for traditional wave equations, especially for high-dimensional problems.

This prompts us to consider how we could save time on calculations. One of the straightforward ideas is to improve the efficiency of solving forward problems. In our previous work, space discretization is based on the finite element method. The computational time grows rapidly as the sizes of discrete problems increase.
Inspired by \cite{JinB2017sub_pod,Zhang2024pod}, we investigate the model reduction low-rank approximation by the proper orthogonal decomposition (POD) method. The so-called data-driven approach is to train a set of POD basis functions based on the observation data of the inverse problem to achieve a significant dimension reduction in the solution space. It does not rely on the prior information of the solution, which avoids the inverse crime.

\begin{problem}\label{prob-back}
We follow the regularization method in \cite{LiZhang2025,LiZhang2024} to numerically reconstruct the unknown initial value $a_1$ based on the terminal value observation data $u(x_i,T)$ contaminated with random noise, where $\{x_i\}_{i=1}^n$ in $ \Omega $. 
 We propose the following problem:

How to design low-rank approximation models based on observational data to accelerate reconstruction without inverse crime?
\end{problem}

Very recently, the authors proposed a well established adjoint-POD method for inverse problems of parabolic type \cite{ZhangZhang2024pod}. It takes snapshots from the adjoint system to construct POD basis functions, which avoids the inverse crime. The point is the equivalence between the span of the snapshots from the forward system and that of its adjoint system. 

Compared to the existing work, our main contributions include:
\begin{enumerate}
    \item The condition for the equivalence between the span of the snapshots from the forward system and that of its adjoint system is relaxed (Lemma 2.2 in \cite{ZhangZhang2024pod}). Specifically, the number of terms $L$ for finite-dimensional truncation does not need to be less than (or equal to) the number of snapshots $M$.
    \item This is a continuation of our previous work \cite{LiZhang2025,LiZhang2024}. The theoretical foundation including stability and regularization methods are investigated in \cite{LiZhang2024}. Rigorous numerical analysis shows the way to balance discretization error, the noise, the regularization parameter, and the number of observation points \cite{LiZhang2025}. This paper extends the interest from stability, regularization, numerical approximation to fast algorithm. It provides a framework from theory to practice for linear inverse problems.
    
\end{enumerate} 
 The structure of the paper is as follows. In Section \ref{frame-forward}, we introduce the matrix form of the  model considered (\ref{eq-gov}), which inspires the discovery that the forward problem and the observation system share the same linear structure in the finite dimensional space. In Section \ref{sec-POD}, the error estimate is derived for the POD approaches for the forward problem with the unknown initial function. In Section \ref{sec-num}, numerical experiments are carried out to verify the theoretical results.

\section{The relationship between the forward problem and the observation system}\label{frame-forward}
A numerical framework is presented in our previous work \cite{LiZhang2025}. Following the idea, we start our fast algorithm investigation in this paper. Let $\nu=\alpha/2$, model (\ref{eq-gov}) is transformed into the following form by using the order reduction method \cite{LiZhang2025}  
\begin{equation}\label{eq-gov-trans}
\begin{cases}
  \partial_t^\nu v - \Delta u = a_1(x)\omega_{2-\alpha}(t), & (x,t) \in \Omega \times (0,T), \\
  v  = \partial_t^\nu u, & (x,t) \in \Omega \times (0,T), \\
  u(x,0) = v(x,0)=0, & x \in \Omega, \\
  u(x,t) = v(x,t)=0, & (x,t) \in \partial \Omega \times (0,T),
\end{cases}
\end{equation}
where $\omega_{2-\alpha}(t)=\frac{t^{1-\alpha}}{\Gamma{(2-\alpha)}}$, the source term $f(x,t)=0$ and the initial value $a_0(x)=0$ are consistent with that in \cite{LiZhang2025,LiZhang2024}.

\subsection{Matrix form of time semi-discrete scheme}
Here, the L1 formula (\ref{L1}), one of the most classical discrete methods, is used to approximate the Caputo derivative at nonuniform meshes $\{t_n|0=t_0<t_1<\dots <t_N=T\}$. To overcome the initial singularity problem, we use graded meshes $t_n=T(n/N)^r$, $r\ge1$ \cite{LiZhang2025}. Denote $\tau_n=t_n-t_{n-1}$, $n\geq1$.
\begin{align}\nonumber
\bar\partial_t^\nu v(t_n)
:&= \sum_{k=1}^{n}\int_{t_{k-1}}^{t_k}\omega_{1-\nu}(t_n-s)\frac{v(t_k)-v(t_{k-1})}{\tau_k}ds\\ \label{L1}
&=\sum_{k=1}^nA_{n-k}^{(n)}\nabla_\tau v(t_k),
\end{align}
where $\nabla_\tau v(t_k)=v(t_k)-v(t_{k-1})$, $A_{n-k}^{(n)}:=\int_{t_{k-1}}^{t_k}\frac{\omega_{1-\nu}(t_n-s)}{\tau_k}ds$. The corresponding numerical scheme (\ref{num-scheme})  is as following:
\begin{equation}\label{num-scheme}
\begin{cases}
  \bar\partial_t^\nu V^n - \Delta U^n = a_1(x)\omega_{2-\alpha}(t_n), & 1\le n \le N, \\
  V^n  = \bar\partial_t^\nu U^n, & 1\le n \le N, \\
  U(x,0) = V(x,0)=0, & x \in \Omega, \\
  U(x,t) = V(x,t)=0, & (x,t) \in \partial \Omega \times (0,T),
\end{cases}
\end{equation}
where $U$ and $V$ are numerical solutions corresponding to $u$ and $v$ in (\ref{eq-gov-trans}). 

Denote the eigenvalue system of the operator $-\Delta$ as $\{\mu_m,\phi_m\}_{m=1}^{\infty}$, where eigenfunctions $\{\phi_m\}_{m=1}^\infty$ forming an orthogonal basis of $L^2(\Omega)$ with eigenvalue $\{\mu_m\}_{m=1}^\infty$. Taking inner product with $\phi_m$ of (\ref{num-scheme}), it arrives that
\begin{equation}\label{num-scheme-c}
\begin{cases}
  \bar\partial_t^\nu V_m^n + \mu_m U_m^n = (a_1)_m\omega_{2-\alpha}(t_n), & 1\le n \le N, \\
  V_m^n  = \bar\partial_t^\nu U_m^n, & 1\le n \le N, 
\end{cases}
\end{equation}
where $\xi_m$ mean $(\xi,\phi_m)_{L^2(\Omega)}$, $\forall$ $\xi$ $\in$ $L^2(\Omega)$.

Let 
$$
\boldsymbol{A}=\left[
\begin{array}{cccc}
    A_0^{(1)}                      &             &          &            \\
   (A_1^{(2)}-A_0^{(2)})           &  A_0^{(2)}  &          &            \\
    \vdots                         & \vdots      & \ddots   &            \\
    (A_{N-1}^{(N)}-A_{N-2}^{(N)})  &\cdots        &\cdots   & A_0^{(N)}  
\end{array}
\right],
$$
$
\boldsymbol{V_m}=\left[
\begin{array}{c}
   V_m^1           \\
   V_m^2           \\
    \vdots         \\
   V_m^N    
\end{array}
\right]
$,
$
\boldsymbol{U_m}=\left[
\begin{array}{c}
   U_m^1           \\
   U_m^2           \\
    \vdots          \\
   U_m^N    
\end{array}
\right]
$,
$
\boldsymbol{\omega}=\left[
\begin{array}{c}
   \omega_{2-\alpha}(t_1)\\
   \omega_{2-\alpha}(t_2)\\
   \vdots\\ 
   \omega_{2-\alpha}(t_N)    
\end{array}
\right]
$,
and
$\boldsymbol{\mu_m}=\mu_m\boldsymbol{I_{N\times N}}$,
where $\boldsymbol{I}$ is the identity matrix. For (\ref{num-scheme-c}), the matrix form of it is as follows:
\begin{equation}\label{matrix-form}
\left[
\begin{array}{cc}
\boldsymbol{A} & \boldsymbol{\mu_m}\\
\boldsymbol{I} & \boldsymbol{-A}
\end{array}
\right]\left[
\begin{array}{c}
\boldsymbol{V_m} \\
\boldsymbol{U_m} 
\end{array}
\right]
=(a_1)_m\left[
\begin{array}{c}
\boldsymbol{\omega} \\
\boldsymbol{0} 
\end{array}
\right].
\end{equation}

\subsection{Data mollification}\label{data-moli}
We consider the observation data $q=u(t_N)+\sigma \epsilon$, $\epsilon$ is generated by standard Gaussian distribution, $\sigma$ denotes the noise level. 
A Tikhonov regularization method is proposed to denoise the observation data $\{q(x_i)\}_{i=1}^m$:
\begin{align*}
q^*=\min\limits_{q^r}\frac{1}{m}\sum_{i=1}^m|q^r(x_i)-q(x_i)|^2+\lambda q^rA(q^r)^T,    
\end{align*}
where $A$ is a positive definite matrix with a regularization parameter $\lambda$. For quasi-uniform distribution data, the estimate of $\frac{1}{m}\sum_{i=1}^m|q^*(x_i)-u(x_i,t_N)|^2$ is given. 
\begin{lemma}(\cite{CenZhang:2025})\label{data-denoise}
When $\rho_1\geq\rho_2\geq\cdots$, of the eigenvalue problem
\begin{align*}
\varphi A^{-1}u=\rho(\varphi,u),~~\forall u\in \mathbb{R}^m,    
\end{align*}
satisfy that $\rho_k\leq Ck^{-\alpha}$, $k=1,2,\cdots,m$ and the corresponding eigenvectors form an orthonormal basis respect to the inner product, one has
\begin{align*}
\mathbb{E}\big[\frac{1}{m}\sum_{i=1}^m|q^*(x_i)-u(x_i,t_N)|^2\big]\leq C\lambda\|q^*\|_{A}^2+C\frac{\sigma^2}{m\lambda^{1/\alpha}},
\end{align*}
where $\|q^*\|_{A}^2=q^*A(q^*)^T$.
\end{lemma}

From Lemma \ref{data-denoise}, we have the optimal choice of the parameter $\lambda^{1+1/\alpha}=O(\sigma^2m^{-1})\|q^*\|_A^{-2}$. It implies that $\mathbb{E}\big[\frac{1}{m}\sum_{i=1}^m|q^*(x_i)-u(x_i,t_N)|^2\big]\rightarrow0$ as the number of observation data $m\rightarrow \infty$. In practical applications, increasing the number of observation points is a natural way to overcome the effects of noise.

For $q^*\in H^k(\Omega)$, its error estimate in $L^2(\Omega)$ can be bounded by the discrete $l^2$-norm, see Lemma \ref{lem-u-un}.
\begin{lemma}\cite[Theorems 3.3 and 3.4]{utreras1988convergence}
\label{lem-u-un}
There exists a constant  $C>0$  such that for all  $  u \in H^k(\Omega)$ with $k>\frac{d}2$ , the following estimates are valid:
\begin{equation*}
\begin{aligned}
\| u \|^2_{L^2(\Omega)} \leq& C \left( \| u \|^2_n + n^{-\frac{2k}{d}} \| u \|^2_{H^k(\Omega)} \right),
\end{aligned}
\end{equation*}
where $\| u \|^2_n:=\frac{1}{n}\sum_{i=1}^n|u(x_i)|^2$.
\end{lemma}

\subsection{The observation system}
For the given final time measurement $q(x)$, we present the observation system (\ref{eq-adjoint}) of the forward problem (\ref{eq-gov-trans}) as follows:
\begin{equation}\label{eq-adjoint}
\begin{cases}
  \partial_t^\nu  \widetilde{v} - \Delta  \widetilde{u} = q(x)\omega_{2-\alpha}(t), & (x,t) \in \Omega \times (0,T), \\
   \widetilde{v}  = \partial_t^\nu  \widetilde{u}, & (x,t) \in \Omega \times (0,T), \\
   \widetilde{u}(x,0) =  \widetilde{v}(x,0)=0, & x \in \Omega, \\
   \widetilde{u}(x,t) =  \widetilde{v}(x,t)=0, & (x,t) \in \partial \Omega \times (0,T).
\end{cases}
\end{equation}
Repeating the process from (\ref{L1}) to (\ref{matrix-form}), we get 
\begin{equation*}
\left[
\begin{array}{cc}
\boldsymbol{A} & \boldsymbol{\mu_m}\\
\boldsymbol{I} & \boldsymbol{-A}
\end{array}
\right]\left[
\begin{array}{c}
\boldsymbol{ \widetilde{V}_m} \\
\boldsymbol{ \widetilde{U}_m} 
\end{array}
\right]
=q_m\left[
\begin{array}{c}
\boldsymbol{\omega} \\
\boldsymbol{0} 
\end{array}
\right].
\end{equation*}
Furthermore, non-singular linear system (\ref{matrix-form}) gives that
\begin{equation}\label{matrix-form-adjoint-1}
\left[
\begin{array}{c}
\boldsymbol{ V_m} \\
\boldsymbol{ U_m} 
\end{array}
\right]=\frac{(a_1)_m}{q_m}\left[
\begin{array}{c}
\boldsymbol{ \widetilde{V}_m} \\
\boldsymbol{ \widetilde{U}_m} 
\end{array}
\right], ~q_m\neq 0.
\end{equation}

The case of noise-free ($\sigma=0$) is discussed firstly.
From \cite{LiZhang2024}, we have
\begin{equation}\label{matrix-form-adjoint-2}
u_m(t_N)=t_NE_{\alpha,2}(-\mu_m t_N^\alpha)(a_1)_m, 
\end{equation}
where $E_{\alpha,2}(-\mu_m t_N^\alpha)=\sum_{n=0}^\infty\frac{(-\mu_mt_N^\alpha)^n}{\Gamma{(n\alpha+2)}}$.
From (\ref{matrix-form-adjoint-2}), we try to give the bounds estimate of $\big|\frac{(a_1)_m}{q_m}\big|$ in (\ref{matrix-form-adjoint-1}). 
\begin{lemma}\label{bounds}
If the noise level $\sigma=0$ and $q_m\neq0$, then there exists constant $C$ such that
$$Cm^{\frac{2}{d}}\le\big|\tfrac{(a_1)_m}{q_m}\big|\le Cm^{\frac{2}{d}},~ m\ge1.$$
\begin{proof}
Based on inequity $|E_{\alpha,2}(-\mu_m t_N^\alpha)|\geq \frac{C}{1+\mu_mt_N^\alpha}$ in \cite{podlubny1999fractional} and the property of eigenvalue $\mu_m\sim m^{\frac{2}{d}}$, $m\ge 1$, one has
\begin{align}\nonumber
\bigg|\frac{(a_1)_m}{q_m}\bigg|
=& \bigg|\frac{1}{t_NE_{\alpha,2}(-\mu_m t_N^\alpha)}\bigg| \\ \nonumber
\leq &C(1+\mu_mt_N^\alpha)\\ \label{matrix-form-adjoint-3}
\leq &Cm^{\frac{2}{d}}.
\end{align}
On the other hand, by using inequity $|E_{\alpha,2}(-\mu_m t_N^\alpha)|\le \frac{C}{1+\mu_mt_N^\alpha}$, we have
$\big|\tfrac{(a_1)_m}{q_m}\big|
\ge Cm^{\frac{2}{d}}$.    
\end{proof}   
\end{lemma}

\begin{lemma}\label{map}By Lemma \ref{bounds}, linear systems (\ref{matrix-form}) and (\ref{matrix-form-adjoint-1}) are discussed in Sobolev spaces.
Denote operator $S:U\rightarrow \widetilde{U}$, if $U\in L^2(\Omega)$, then $\widetilde{U}\in H^2(\Omega)\cap H_0^1(\Omega)$.
\begin{proof}
\begin{align*}
\|U\|_{L^2}^2:=\sum_{m=1}^\infty|U_m|^2=\sum_{m=1}^\infty\big|\tfrac{(a_1)_m}{q_m}\big|^2|\widetilde{U}_m|^2\sim\sum_{m=1}^\infty m^{\frac{4}{d}}|\widetilde{U}_m|^2\sim\|\widetilde{U}\|_{H^2}^2.
\end{align*}   
The proof is completed by the definition of the eigenvalue system of the operator $-\Delta$ and Lemma \ref{bounds}.
\end{proof}   
\end{lemma}

\begin{remark}\label{bound-noise}
Since $q(x)$ is given in the continuous space $\Omega$, we could take enough observation points to perform data mollification so that the factor $\frac{(a_1)_m}{q_m}$ satisfies the Lemma \ref{bounds}. Then Lemma \ref{map} also holds for noise data. 
\end{remark}

\iffalse
In practical calculations, extracting the singular term $\omega_{2-\alpha}(t)$ could improve the accuracy of the algorithm. Since the final time measurement without noise is a smooth function, scheme (\ref{eq-adjoint-trans}) is prefer.
\begin{remark}
The adjoint system (\ref{eq-adjoint}) can be transfer into the following form:
\begin{equation}\label{eq-adjoint-trans}
\begin{cases}
  \partial_t^\nu \textbf{v}  - \Delta \textbf{u} = t\Delta q(x), & (x,t) \in \Omega \times (0,T), \\
  \textbf{v}  = \partial_t^\nu \textbf{u}, & (x,t) \in \Omega \times (0,T), \\
  \textbf{u}(x,0) = \textbf{v}(x,0)=0, & x \in \Omega, \\
  \textbf{u}(x,t) = \textbf{v}(x,t)=0, & (x,t) \in \partial \Omega \times (0,T),
\end{cases}
\end{equation}
where $\widetilde{u}=\textbf{u}+tq(x)$. For more details see \cite{LiZhang2025}.
\end{remark}
\fi

In the following, we show that the solution could be approximated by a finite-dimensional truncation. We take $m=1,\cdots,L$ in (\ref{matrix-form-adjoint-1}), then the numerical solutions of the forward problem and the observation system share the same finite-dimensional linear space.
\begin{remark}\label{re-eq} From the well-posedness of (\ref{eq-gov}) \cite{LiZhang2025,LiZhang2024}, we know that $u(\cdot,t)\in H^2(\Omega)$, $t\in \mathbb{R}^+$.
Let $u^L:=\sum_{m=1}^{L}u_m\phi_m$, we have 
$$\|u-u^L\|_{L^2}^2=\sum_{m=L+1}^{\infty}|u_m|^2\leq \mu_{L+1}^{-2}\sum_{m=1}^{\infty}\mu_m^2|u_m|^2=\mu_{L+1}^{-2}\|\Delta u\|_{L^2}^2,$$
where $\mu_{L+1}\sim (L+1)^{\frac{2}{d}}$. 
Similarly, we have $\|v-v^L\|_{L^2}^2\leq \mu_{L+1}^{-2}\|\Delta v\|_{L^2}^2$.
The result also holds for the solution of the observation system.
\end{remark}

\section{Proper orthogonal decomposition}\label{sec-POD}
In this section, we investigate the approximation properties of solutions in the POD space generated by the observation system.
Our aim is to derive an error estimate for the solution of the forward problem and its corresponding POD space that snapshots come from the observation system (\ref{eq-adjoint}) but not the forward problem. 

First, we recall the general framework of POD. For $N\in\mathbb{N}$, let $\{y_n\}_{n=1}^N\subset X:=H_0^1(\Omega)$ be an ensemble of snapshots. The POD basis functions $\{\psi_j\}_{j=1}^r$ are constructed by minimizing the following projection error:
\begin{align*}
 \frac{1}{N}\sum_{n=1}^N\|y_n-\sum_{j=1}^r(y_n,\psi_j)_X\psi_j\|_{X}^2.   
\end{align*}
On the other hand, it can be reduced to the following eigenvalue problem:
$$Kv=\lambda v,$$
where the correlation matrix $K$ is calculated from the snapshots $\{y_n\}_{n=1}^N$ with the entries $K_{ij}=\frac{1}{N}(y_i,y_j)_X$, and $K$ is symmetric and semi-positive definite. The corresponding approximation error is presented \cite{Sirovich1987}.

\begin{lemma}
Let $\lambda_1\ge \lambda_2\ge \cdots\ge \lambda_r>0$ be the positive eigenvalues of the correlation matrix $K$ and $v_1,\cdots,v_r\in\mathbb{R}^N$ be the corresponding orthonormal eigenvectors. Then a POD basis of rank $m\le r$ is given by
$$\psi_j=\frac{1}{\sqrt{\lambda_j}}\sum_{n=1}^N(v_j)^ny_n,$$
where $(v_j)^n$ denotes the $n$-th component of the eigenvector $v_j$. Moreover, the error is given by
\begin{align*}
 \frac{1}{N}\sum_{n=1}^N\|y_n-\sum_{j=1}^r(y_n,\psi_j)_X\psi_j\|_{X}^2=\sum_{j=m+1}^r\lambda_j.   
\end{align*}
\end{lemma}

\subsection{POD approach for the observation system}
For the sake of completeness, we first introduce the finite element method and then give the relevant POD scheme.
A fully discrete scheme is as follows: find $\widetilde{U}_h^n$, $\widetilde{V}_h^n\in X_h$ for $n=1$, 2, $\cdots$, $N$
\begin{equation}\label{num-scheme-fed}
\begin{cases}
  (\bar\partial_t^\nu \widetilde{V}_h^n,\varphi) + (\nabla \widetilde{U}_h^n,\nabla \varphi) = (q(x)\omega_{2-\alpha}(t_n),\varphi), & \forall \varphi\in X_h, \\
  (\nabla\widetilde{V}_h^n,\nabla\varphi)  = (\bar\partial_t^\nu \nabla\widetilde{U}_h^n,\nabla\varphi),
\end{cases}
\end{equation}
where $X_h$ is the associated continuous piecewise linear finite element space. 
For model reduction, the POD methodology is adopted to (\ref{num-scheme-fed}). Taking fully discrete solutions $\{\widetilde{U}_h^n\}_{n=1}^N$ and fractional difference quotients $\{\bar\partial_t^\nu \widetilde{U}_h^n\}_{n=1}^N$, $\{\bar\partial_t^\nu \widetilde{V}_h^n\}_{n=1}^N$ as snapshots to generate an optimal orthonormal basis $\{\psi_j\}_{j=1}^r$. A numerical scheme is obtained using the POD space $X_h^m$, $m\le r$, spanned by the first $m$ POD basis functions. Find $U_m^n\in X_h^m$, $n=1,2$, $\cdots$, $N$ such that
\begin{equation}\label{num-scheme-pod}
\begin{cases}
  (\bar\partial_t^\nu \widetilde{V}_m^n,\varphi) + (\nabla \widetilde{U}_m^n,\nabla \varphi) = (q(x)\omega_{2-\alpha}(t_n),\varphi), & \forall \varphi\in X_h^m, \\
  (\nabla\widetilde{V}_m^n,\nabla\varphi)  = (\bar\partial_t^\nu \nabla\widetilde{U}_m^n,\nabla\varphi).
\end{cases}
\end{equation}

Denoting $e_m^n=\widetilde{U}_h^n-\widetilde{U}_m^n$, $z_m^n=\widetilde{V}_h^n-\widetilde{V}_m^n$, we give the error estimate for (\ref{num-scheme-pod}). Some lemmas are introduced.

\begin{lemma}\label{FD-ineq}
(\cite{LiaoH2018L1})For $V^n$, $1\leq n\leq N$, one has
\begin{align*}
(\bar\partial_t^\nu V^n,V^n)\geq \frac{1}{2}\bar\partial_t^\nu \|V^n\|_{L^2}^2.    
\end{align*}
\end{lemma}

\begin{lemma}\label{grownwall}
(\cite{LyuP2022SFOR})Let $(g^n)_{n=1}^N$ and $(\lambda_l)_{l=0}^{N-1}$ be given nonnegative sequences. Assume that there exists a constant $\Lambda$ such that $\Lambda\geq\sum_{l=0}^{N-1}\lambda_l$, and that the maximum step satisfies 
$$\max_{1\leq n \leq N}\tau_n\leq \frac{1}{\sqrt[\nu]{\Gamma{(2-\nu)}\Lambda}}.$$
Then, for any nonnegative sequence $(v^k)_{k=0}^N$ and $(w^k)_{k=0}^N$ satisfying
$$\sum_{k=1}^nA_{n-k}^{(n)}\nabla_\tau\big[(v^k)^2+(w^k)^2\big]\leq\sum_{k=1}^n\lambda_{n-k}\big(v^k+w^k\big)^2+(v^n+w^n)g^n,~~1\leq n\leq N,$$
it holds that
$$v^n+w^n\leq 4E_\nu(4\Lambda t_n^\nu)\bigg(v^0+w^0+\max_{1\leq k\leq n}\sum_{j=1}^kP_{k-j}^{(k)}g^j\bigg),~~1\leq n\leq N,$$
where $E_\nu(z)=\sum_{k=0}^\infty\frac{z^k}{\Gamma{(1+k\nu)}}$ is the Mittag-Leffler function.
\end{lemma}

\begin{lemma}\label{P}
For the sequence $(P_{n-j}^{(n)})_{j=1}^n$, some properties are given in \cite{LiaoH2018L1}.
\begin{align*}
&\sum_{j=k}^nP_{n-j}^{(n)}A_{j-k}^{(j)}\equiv1,~~1\leq k\leq n,\\
&0\leq P_{n-j}^{(n)}\leq \Gamma{(2-\nu)}\tau_j^\nu,~~\sum_{j=1}^nP_{n-j}^{(n)}t_j^{-\alpha/2}\leq C,~~ 1\leq j\leq n\leq N.  
\end{align*}
\end{lemma}

\begin{theorem}\label{error-adj}
Let $\widetilde{U}_h^n$ and $\widetilde{U}_m^n$ be the solutions of (\ref{num-scheme-fed}) and (\ref{num-scheme-pod}). Then there holds 
\begin{align*} 
\frac{1}{N}\sum_{n=1}^N\|\widetilde{U}_h^n-\widetilde{U}_m^n\|_{L^2}^2\leq CN^{2-\alpha}\sum_{j=m+1}^r\lambda_j.    
\end{align*} 
\end{theorem}
\begin{proof}
Introducing the Ritz projection operator $R_n^m:X_h\rightarrow X_h^m$, such that $(\nabla R_n^m \widetilde{U}_h^n, \nabla \varphi_m)=(\nabla \widetilde{U}_h^n, \nabla \varphi_m)$, $\forall \varphi_m\in X_h^m\subset X_h$.  Splitting $e_m^n=(\widetilde{U}_h^n-R_h^m\widetilde{U}_h^n)+(R_h^m\widetilde{U}_h^n-\widetilde{U}_m^n):=\rho_u^n+\theta_u^n$, $z_m^n=(\widetilde{V}_h^n-R_h^m\widetilde{V}_h^n)+(R_h^m\widetilde{V}_h^n-\widetilde{V}_m^n):=\rho_v^n+\theta_v^n$. 
First, we give an estimate for $\rho_u^n$.
\begin{align}\label{pod-1}
\frac{1}{N}\sum_{n=1}^N\|\rho_u^n\|_{L^2}^2\leq \frac{1}{N}\sum_{n=1}^N\|\rho_u^n\|_{H_0^1}^2\leq C\sum_{j=m+1}^r\lambda_j.    
\end{align}
Next, we derive an estimate for $\theta_u^n$.
\begin{align*}
&(\bar\partial_t^\nu\theta_v^n,\varphi_m)+(\nabla\theta_u^n,\nabla\varphi_m)\\
=&(\bar\partial_t^\nu R_h^m\widetilde{V}_h^n,\varphi_m)+(\nabla R_h^m\widetilde{U}_h^n,\nabla\varphi_m)-(\bar\partial_t^\nu \widetilde{V}_m^n,\varphi_m)-(\nabla \widetilde{U}_m^n,\nabla\varphi_m)\\
=&(\bar\partial_t^\nu R_h^m\widetilde{V}_h^n,\varphi_m)+(\nabla \widetilde{U}_h^n,\nabla\varphi_m)-(q(x)\omega_{2-\alpha}(t_n),\varphi_m)\\
=&(\bar\partial_t^\nu (R_h^m\widetilde{V}_h^n-\widetilde{V}_h^n),\varphi_m)\\
:=&-(\bar\partial_t^\nu \rho_v^n,\varphi_m).
\end{align*}
Substituting $\varphi_m=\theta_v^n$ into the above equation, it gives that
\begin{align}\label{pod-2}
&(\bar\partial_t^\nu\theta_v^n,\theta_v^n)+(\nabla\theta_u^n,\nabla\theta_v^n)=-(\bar\partial_t^\nu \rho_v^n,\theta_v^n).
\end{align}
Furthermore, one has
\begin{align*}
(\nabla \theta_v^n,\nabla\varphi_m)
&=(\nabla R_h^m\widetilde{V}_h^n,\nabla\varphi_m)-(\nabla \widetilde{V}_m^n,\nabla\varphi_m)\\
&:=(\nabla \widetilde{V}_h^n,\nabla\varphi_m)-(\bar\partial_t^\nu \nabla\widetilde{U}_m^n,\nabla\varphi_m)\\
&=(\bar\partial_t^\nu \nabla(\widetilde{U}_h^n-\widetilde{U}_m^n),\nabla\varphi_m)\\
&=(\bar\partial_t^\nu \nabla (\rho_u^n+\theta_u^n),\nabla\varphi_m).
\end{align*}
Substituting $\varphi_m=\theta_u^n$ into the above equation, we have
\begin{align}\label{pod-3}
(\nabla \theta_v^n,\nabla\theta_u^n) 
=(\bar\partial_t^\nu \nabla \rho_u^n,\nabla\theta_u^n)+(\bar\partial_t^\nu \nabla \theta_u^n,\nabla\theta_u^n).
\end{align}
For (\ref{pod-3}), from (\ref{pod-2}), we get
\begin{align}\label{pod-4}
(\bar\partial_t^\nu\theta_v^n,\theta_v^n)+(\bar\partial_t^\nu \nabla \theta_u^n,\nabla\theta_u^n)=-(\bar\partial_t^\nu \nabla \rho_u^n,\nabla\theta_u^n)-(\bar\partial_t^\nu \rho_v^n,\theta_v^n).
\end{align}
From Lemma \ref{FD-ineq}, for (\ref{pod-4}), we have
\begin{align*}
\frac12\bar\partial_t^\nu\|\theta_v^n\|_{L^2}^2+\frac12\bar\partial_t^\nu \|\nabla \theta_u^n\|_{L^2}^2
&\leq \|\bar\partial_t^\nu \nabla \rho_u^n\|_{L^2}\|\nabla\theta_u^n\|_{L^2}+\|\bar\partial_t^\nu \rho_v^n\|_{L^2}\|\theta_v^n\|_{L^2}\\
&\leq (\|\bar\partial_t^\nu \nabla \rho_u^n\|_{L^2}+\|\bar\partial_t^\nu \rho_v^n\|_{L^2})(\|\nabla\theta_u^n\|_{L^2}+\|\theta_v^n\|_{L^2}).
\end{align*}
Lemma \ref{grownwall} gives that 
\begin{align*}
\|\theta_u^n\|_{L^2}\leq C\big(\|\theta_v^n\|_{L^2}+ \|\nabla \theta_u^n\|_{L^2}\big) 
&\leq C\max_{1\leq k\leq n}\sum_{j=1}^kP_{k-j}^{(k)}(\|\bar\partial_t^\nu \nabla \rho_u^j\|_{L^2}+\|\bar\partial_t^\nu \rho_v^j\|_{L^2}),~1\leq n\leq N.
\end{align*}
Suppose $\max_{1\leq k\leq n}\sum_{j=1}^kP_{k-j}^{(k)}(\|\bar\partial_t^\nu \nabla \rho_u^j\|_{L^2}+\|\bar\partial_t^\nu \rho_v^j\|_{L^2})=\sum_{j=1}^{n_0}P_{n_0-j}^{({n_0})}(\|\bar\partial_t^\nu \nabla \rho_u^j\|_{L^2}+\|\bar\partial_t^\nu \rho_v^j\|_{L^2})$, $n_0\leq n$, from Lemma \ref{P}, one has
\begin{align*}
\|\theta_u^n\|_{L^2}^2
&\leq C\bigg(\sum_{j=1}^{n_0}P_{n_0-j}^{({n_0})}(\|\bar\partial_t^\nu \nabla \rho_u^j\|_{L^2}+\|\bar\partial_t^\nu \rho_v^j\|_{L^2})\bigg)^2\\
&\leq C\bigg(\sum_{j=1}^{n_0}\tau_j^{\alpha/2}(\|\bar\partial_t^\nu \nabla \rho_u^j\|_{L^2}+\|\bar\partial_t^\nu \rho_v^j\|_{L^2})\bigg)^2\\
&\leq C\bigg(\sum_{j=1}^{n}\tau_j^{\alpha/2}(\|\bar\partial_t^\nu \nabla \rho_u^j\|_{L^2}+\|\bar\partial_t^\nu \rho_v^j\|_{L^2})\bigg)^2\\
&\leq C\sum_{j=1}^{n}\tau_j^{\alpha}\sum_{j=1}^{n}\bigg(\|\bar\partial_t^\nu \nabla \rho_u^j\|_{L^2}+\|\bar\partial_t^\nu \rho_v^j\|_{L^2}\bigg)^2\\
&\leq CN\sum_{j=m+1}^r\lambda_j\sum_{j=1}^{n}\tau_j^{\alpha}\\
&\leq CnN^{1-\alpha}\sum_{j=m+1}^r\lambda_j.
\end{align*}
Then, we have
\begin{equation}\label{pod-5}
\frac{1}{N}\sum_{n=1}^N\|\theta_u^n\|_{L^2}^2
\leq  CN^{-\alpha}\sum_{j=m+1}^r\lambda_j\sum_{n=1}^Nn
\leq  CN^{2-\alpha}\sum_{j=m+1}^r\lambda_j.  
\end{equation}
Combining (\ref{pod-1}) and (\ref{pod-5}), the error between $e_m^n=\widetilde{U}_h^n-\widetilde{U}_m^n$ is given
\begin{align*}
\frac{1}{N}\sum_{n=1}^N\|\widetilde{U}_h^n-\widetilde{U}_m^n\|_{L^2}^2\leq  \frac{2}{N}\sum_{n=1}^N(\|\rho_u^n\|_{L^2}^2+\|\theta_u^n\|_{L^2}^2)\leq CN^{2-\alpha}\sum_{j=m+1}^r\lambda_j.    
\end{align*}
\end{proof}

\subsection{POD approach for the forward problem}
Recall Remark \ref{re-eq}, we consider the solution of the forward system approximated by a finite-dimensional truncation. For $U_h^n \approx\sum_{k=1}^L(U_h^n,\phi_k)\phi_k$, there exists a positive constant $\varepsilon$ such that $\|U_h^n-\sum_{k=1}^L(U_h^n,\phi_k)\phi_k\|_{L^2}<\epsilon$, $1\le n \le N$. Under the assumption that $\varepsilon$ is small enough, the error estimate between the solution in the finite element space (\ref{num-scheme-fed-for}) and its in the POD space (\ref{num-scheme-pod-for}) is derived following the idea of Theorem \ref{error-adj}. 
A fully discrete scheme is as follows: find $U_h^n$, $V_h^n\in X_h$ for $n=1$, 2, $\cdots$, $N$
\begin{equation}\label{num-scheme-fed-for}
\begin{cases}
  (\bar\partial_t^\nu V_h^n,\varphi) + (\nabla U_h^n,\nabla \varphi) = (a_1(x)\omega_{2-\alpha}(t_n),\varphi), & \forall \varphi\in X_h, \\
  (\nabla V_h^n,\nabla\varphi)  = (\bar\partial_t^\nu \nabla U_h^n,\nabla\varphi),
\end{cases}
\end{equation}
where $X_h$ is the associated continuous piecewise linear finite element space. 
The corresponding numerical scheme is obtained using the POD space $X_h^m$, $m\le r$, spanned by the first $m$ POD basis functions. Find $U_m^n\in X_h^m$, $n=1,2$, $\cdots$, $N$ such that
\begin{equation}\label{num-scheme-pod-for}
\begin{cases}
  (\bar\partial_t^\nu V_m^n,\varphi) + (\nabla U_m^n,\nabla \varphi) = (a_1(x)\omega_{2-\alpha}(t_n),\varphi), & \forall \varphi\in X_h^m, \\
  (\nabla V_m^n,\nabla\varphi)  = (\bar\partial_t^\nu \nabla U_m^n,\nabla\varphi).
\end{cases}
\end{equation}

\begin{remark}
The POD space $X_h^m$ in (\ref{num-scheme-pod-for}) comes from (\ref{num-scheme-fed}).  Specifically, fully discrete solutions $\{\widetilde{U}_h^n\}_{n=1}^N$ and fractional difference quotients $\{\bar\partial_t^\nu \widetilde{U}_h^n\}_{n=1}^N$, $\{\bar\partial_t^\nu \widetilde{V}_h^n\}_{n=1}^N$ are taken as snapshots to generate an optimal orthonormal basis $\{\psi_j\}_{j=1}^r$.
\end{remark}

\begin{theorem}\label{main-th1}
Let $U_h^n$ and $U_m^n$ be the solutions of (\ref{num-scheme-fed-for}) and (\ref{num-scheme-pod-for}). Then there holds 
\begin{align*} 
\frac{1}{N}\sum_{n=1}^N\|U_h^n-U_m^n\|_{L^2}^2\leq CL^{4/d+2}N^{2-\alpha}\sum_{j=m+1}^r\lambda_j.    
\end{align*}    
\end{theorem}
\begin{proof} Repeating the process in Theorem \ref{error-adj}, we have
\begin{align}\label{th2-1}
\frac{1}{N}\sum_{n=1}^N\|U_h^n-U_m^n\|_{L^2}^2\leq  \frac{C}{N}\sum_{n=1}^N \|U_h^n-R_h^mU_h^n\|_{L^2}^2+\frac{C}{N}\sum_{n=1}^N \|R_h^mU_h^n-U_m^n\|_{L^2}^2.
\end{align}
Let $I_1:=\frac{C}{N}\sum_{n=1}^N \|U_h^n-R_h^mU_h^n\|_{L^2}^2$ and $I_2:=\frac{C}{N}\sum_{n=1}^N \|R_h^mU_h^n-U_m^n\|_{L^2}^2$.
For the term $I_1$, from Remark \ref{bound-noise}, one has
\begin{align}
\frac{C}{N}\sum_{n=1}^N \|U_h^n-R_h^mU_h^n\|_{L^2}^2\nonumber
&\leq \frac{C}{N}\sum_{n=1}^N \bigg\|(I-R_h^m)\sum_{k=1}^L(U_h^n,\phi_k)\phi_k\bigg\|_{L^2}^2\\ \nonumber
&\leq \frac{C}{N}\sum_{n=1}^N \bigg\|(I-R_h^m)\sum_{k=1}^L\frac{(a_1)_k}{q_k}(\widetilde{U}_h^n,\phi_k)\phi_k\bigg\|_{L^2}^2\\ \nonumber
&\leq \frac{C}{N}\sum_{n=1}^N \max_k\bigg|\frac{(a_1)_k}{q_k}\bigg|^2L\sum_{k=1}^L\bigg\|(I-R_h^m)(\widetilde{U}_h^n,\phi_k)\phi_k\bigg\|_{L^2}^2\\ \nonumber
&\leq \frac{C}{N}L^{4/d+1}\sum_{n=1}^N \sum_{k=1}^L\bigg\|(I-R_h^m)(\widetilde{U}_h^n,\phi_k)\phi_k\bigg\|_{L^2}^2\\ \nonumber
&\leq CL^{4/d+1} \sum_{k=1}^L\frac{1}{N}\sum_{n=1}^N\bigg\|(I-R_h^m)(\widetilde{U}_h^n,\phi_k)\phi_k\bigg\|_{L^2}^2\\ \label{pod-6}
&\leq CL^{4/d+2}\sum_{j=m+1}^r\lambda_j.
\end{align}

Similar to (\ref{pod-2})-(\ref{pod-5}) and (\ref{pod-6}), for $I_2$, we have
\begin{align*}
\|R_h^mU_h^n-U_m^n\|_{L^2}^2
&\leq C\bigg(\sum_{j=1}^{n}\tau_j^{\alpha/2}\big(\|\bar\partial_t^\nu \nabla (U_h^j-R_h^mU_h^j)\|_{L^2}+\|\bar\partial_t^\nu (V_h^j-R_h^mV_h^j)\|_{L^2}\big)\bigg)^2\\
&\leq C\sum_{j=1}^{n}\tau_j^{\alpha}\sum_{j=1}^{n}\bigg(\|\bar\partial_t^\nu \nabla (U_h^j-R_h^mU_h^j)\|_{L^2}+\|\bar\partial_t^\nu (V_h^j-R_h^mV_h^j)\|_{L^2}\bigg)^2\\
&\leq CL^{4/d+2}N\sum_{j=m+1}^r\lambda_j\sum_{j=1}^{n}\tau_j^{\alpha}\\
&\leq CL^{4/d+2}nN^{1-\alpha}\sum_{j=m+1}^r\lambda_j,
\end{align*}
then it gives
\begin{align*}
\frac{1}{N}\sum_{n=1}^N\|R_h^mU_h^n-U_m^n\|_{L^2}^2\leq  CL^{4/d+2}N^{2-\alpha}\sum_{j=m+1}^r\lambda_j.
\end{align*}
\end{proof}

\begin{remark}
Theorem \ref{main-th1} illustrates that we take the snapshots from the observation system (\ref{eq-adjoint}) based on the known final time measurement $q(x)$ to construct a POD model reduction space that still is a good low-rank approximation for the solutions of the forward system.
\end{remark}

\section{Numerical algorithm for the backward problem}\label{sec-num}
In this section, we propose an algorithm for quickly reconstructing initial values based on data-driven ideas. The core lies in training a low-rank approximation space(POD space) for the forward system,  based on observational data. Reconstructing the initial values in the POD space can improve computational efficiency.
\\
\textbf{ First step.} Perform quasi-uniform sampling on the observations of $q(x)$ over the entire region to obtain discrete observational data $\{q(x_i)\}_{i=1}^m$.\\ 
\textbf{ Second step.} Perform data mollification, we obtain regularized data $q^r$(For details, refer to section \ref{data-moli}). \\
\textbf{ Third step.} Solve the observation system with $q^r$ (\ref{num-scheme-fed}) to obtain snapshots, and construct a POD space $X_h^m$.\\
\textbf{ Fourth step.} Solve the following minimization problem (\ref{mini-pro}):
\begin{align}\label{mini-pro}
a_{pod}=\arg \min\limits_{a_1\in X_h^m}\frac{1}{m}\sum_{i=1}^m|(Sa_1)(x_i)-q(x_i)|^2+\lambda\|a_1\|_{H^1}^2,
\end{align}
where $S$ is the forward operator $Sa_1\rightarrow u(T)$.

Unlike \cite{LiZhang2025,LiZhang2024}, we recover the initial function in a low-dimensional intrinsic space based on the model but not in the finite element space. The choice of $\lambda$, we refer to \cite{LiZhang2025}.

\begin{remark}
In previous sections, we give the estimates for POD approaches. Specifically, we take fully discrete solutions $\{\widetilde{U}_h^n\}_{n=1}^N$ and fractional difference quotients $\{\bar\partial_t^\nu \widetilde{U}_h^n\}_{n=1}^N$, $\{\bar\partial_t^\nu \widetilde{V}_h^n\}_{n=1}^N$ as snapshots to generate an optimal orthonormal basis $\{\psi_j\}_{j=1}^r$.    
In numerical experiments, we find that the POD space is good enough when only taking $\{\widetilde{U}_h^n\}_{n=1}^N$ as snapshots in space $X\in L^2(\Omega)$.
\end{remark}

\begin{example}\label{ex1}We consider the following equation with $a_1(x)=\sin(x)$:
\begin{equation}\label{ex1_eq}
\begin{cases}
  \partial_t^\nu v - \Delta u = a_1(x)\omega_{2-\alpha}(t), & (x,t) \in (0,\pi) \times (0,0.1), \\
  v  = \partial_t^\nu u, & (x,t) \in (0,\pi) \times (0,0.1), \\
  u(x,0) = v(x,0)=0, & x \in (0,\pi), \\
  u(x,t) = v(x,t)=0, & (x,t) \in \{0,\pi\}  \times (0,0.1).
\end{cases}
\end{equation}    
In this example, we first solve the problem (\ref{ex1_eq}) using the FEM. Then, we compare the numerical result with the use of the proper orthogonal decomposition method to solve (\ref{ex1_eq}).
\end{example}
The time grid $t_n=0.1(\frac{n}{N})^r$, $N=400$, $r=\frac{2-\nu}{1-\nu}$,  $\nu=\frac{3}{4}$, the space grid $x_i=ih$, $h=\frac{\pi}{200}$.
We take five POD basis functions in the calculation. The maximum absolute error of FEM and POD is around $2\times 10^{-16}$.
The time costs are $0.379$ seconds and $0.026$ seconds, respectively. It implies that the POD method can solve a forward problem in only one-tenth of the time required by the FEM while maintaining almost the same accuracy.

\begin{figure}[!htb]
        \begin{minipage}{0.48\linewidth}% 表示图片的占用那一列的宽度
		\centering
		%\vspace{-0.6cm}%表示图片与最上方的文字的距离
		\setlength{\abovecaptionskip}{0.28cm}% 表示caption 与图片之间的距离
		\includegraphics[width=1\linewidth]{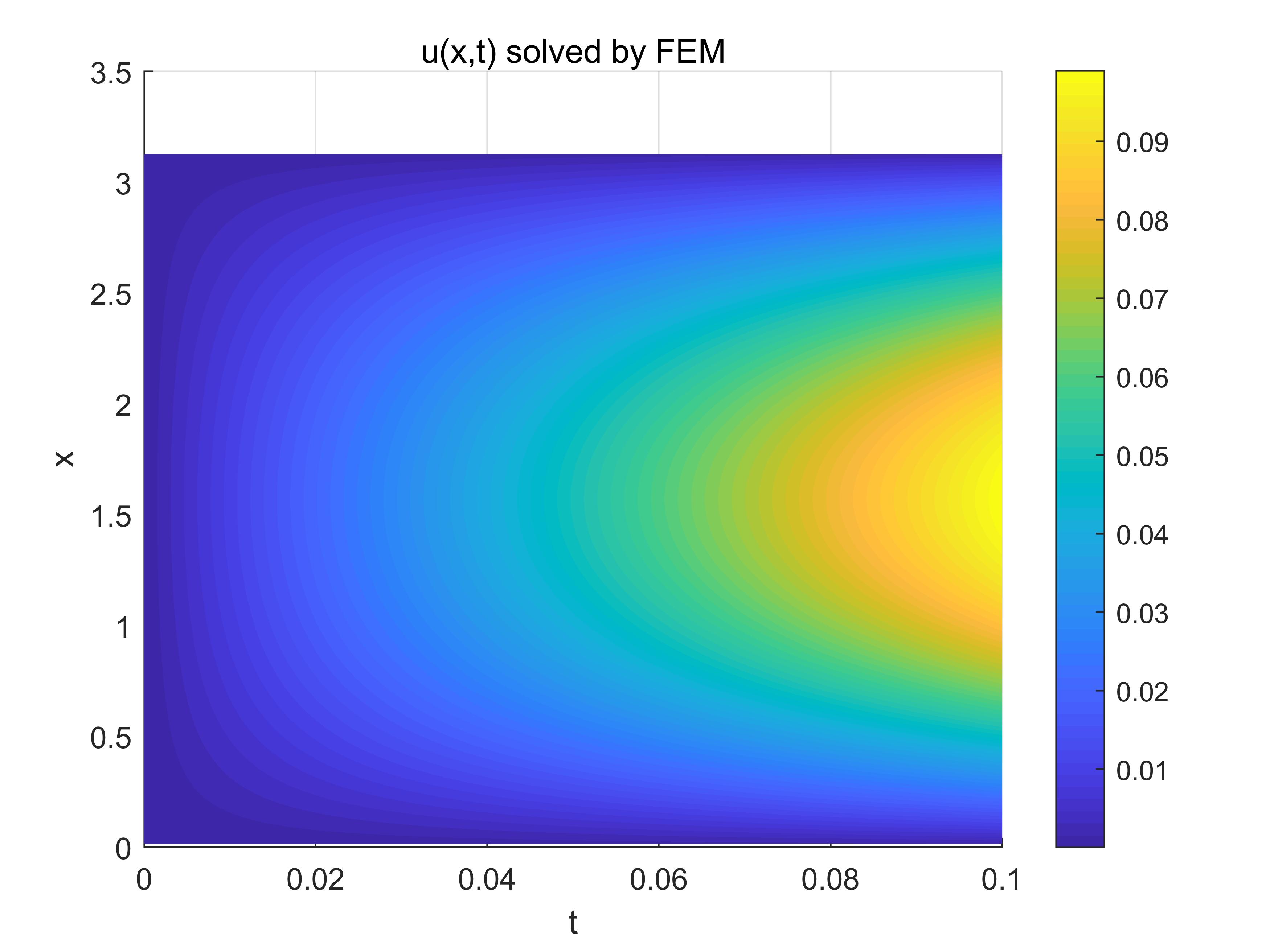}
		\caption{$u(x,t)$ (FEM).}
	\end{minipage}
        \begin{minipage}{0.48\linewidth}% 表示图片的占用那一列的宽度
		\centering
		%\vspace{-0.6cm}%表示图片与最上方的文字的距离
		\setlength{\abovecaptionskip}{0.28cm}% 表示caption 与图片之间的距离
		\includegraphics[width=1\linewidth]{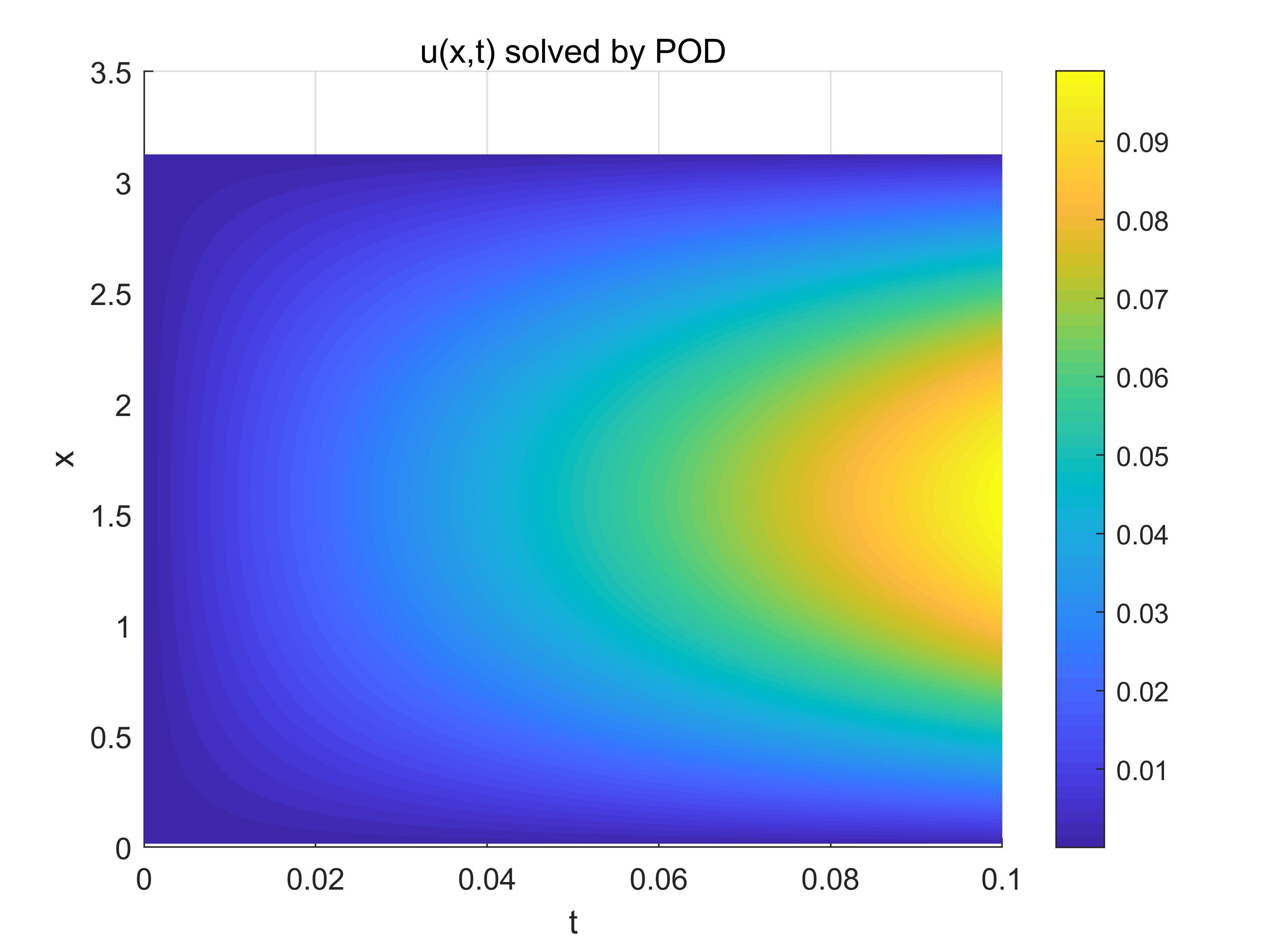}
		\caption{$u(x,t)$ (POD).}
	  \end{minipage}    
\end{figure}

\begin{figure}[!htb]       
		\centering
		\setlength{\abovecaptionskip}{0.28cm}% 表示caption 与图片之间的距离
		\includegraphics[width=0.48\linewidth]{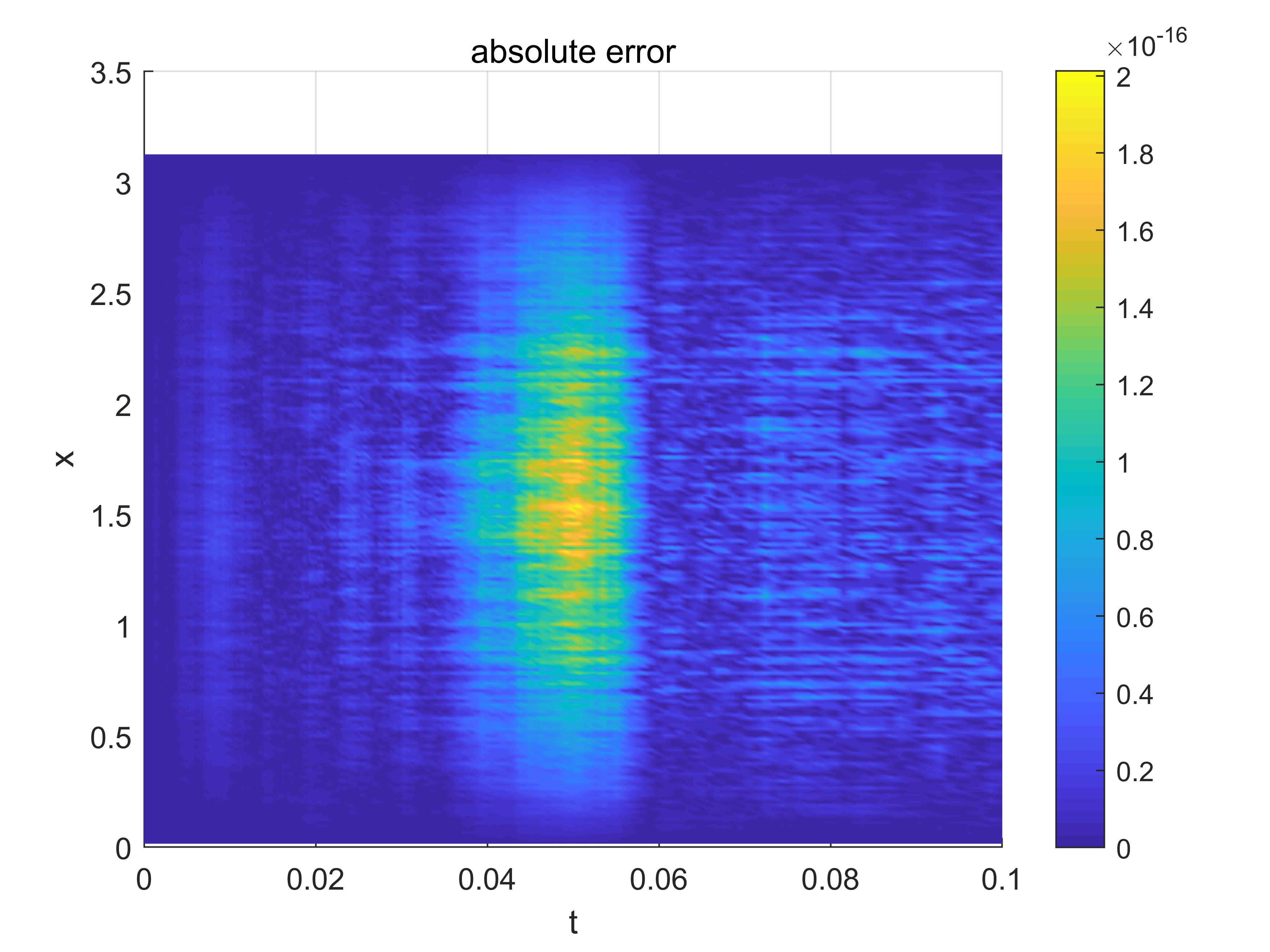}
		\caption{Absolute error.} 
\end{figure}

\begin{example}\label{ex2}
Let $q(x)=u(x,0.1)$, where $u(x,0.1)$ is the terminal solution of (\ref{ex1_eq}) with $\nu=\frac34$.
We solve the following equation by FEM to get snapshots of the POD space:
\begin{equation*}
\begin{cases}
  \partial_t^\nu v - \Delta u = q(x)\omega_{2-\alpha}(t), & (x,t) \in (0,\pi) \times (0,0.1), \\
  v  = \partial_t^\nu u, & (x,t) \in (0,\pi) \times (0,0.1), \\
  u(x,0) = v(x,0)=0, & x \in (0,\pi), \\
  u(x,t) = v(x,t)=0, & (x,t) \in \{0,\pi\}  \times (0,0.1).
\end{cases}
\end{equation*}  
Then, we intend to find $a_1(x)$ in Example \ref{ex1} based on the terminal observation $q(x)$. The grid settings are the same as for Example \ref{ex1}. We also take five POD basis functions in the calculation.
\end{example}

The time costs are $10.08$ seconds for FEM \cite{LiZhang2025} and $0.55$ seconds for POD. Since it is almost a noise-free case, the regularization parameter $\lambda=0$ in (\ref{mini-pro}). Both methods obtain the satisfied initial function $a_1(x)=\sin(x)$.

Then, we consider the case of noise data for Example \ref{ex2}. The noise level $\epsilon=\frac{\sigma}{\|u(T)\|_\infty}\approx \frac{0.015}{0.1}=15\%$. The regularization parameter $\lambda=5.43\times10^{-6}$. The time costs are $90.52$ seconds for FEM \cite{LiZhang2025} and $1.41$ seconds for POD.

\begin{figure}[!htb]
        \begin{minipage}{0.48\linewidth}% 表示图片的占用那一列的宽度
		\centering
		%\vspace{-0.6cm}%表示图片与最上方的文字的距离
		\setlength{\abovecaptionskip}{0.28cm}% 表示caption 与图片之间的距离
		\includegraphics[width=1\linewidth]{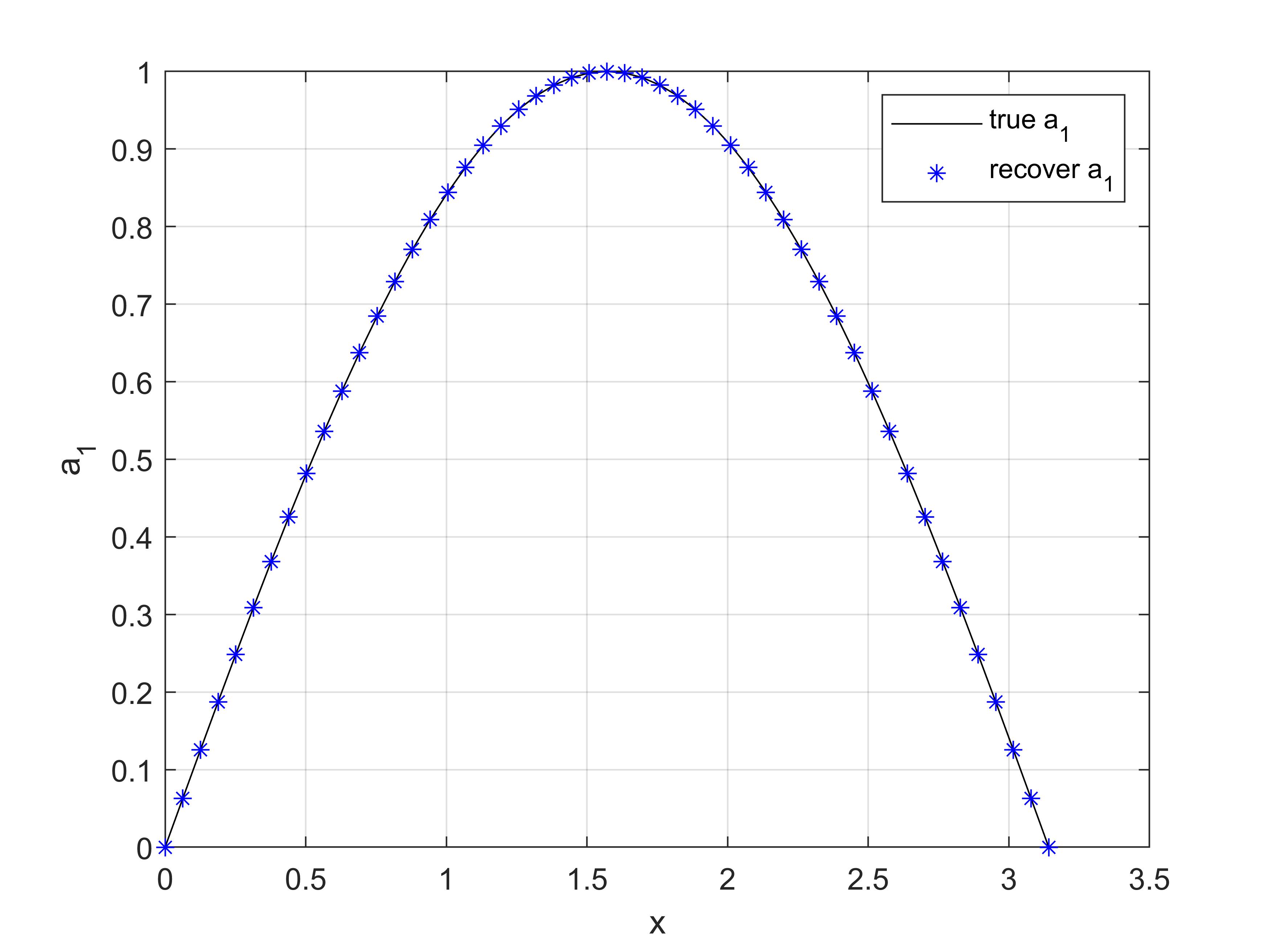}
		\caption{Recover $a_1(x)$ \cite{LiZhang2025}.}
	\end{minipage}
        \begin{minipage}{0.48\linewidth}% 表示图片的占用那一列的宽度
		\centering
		%\vspace{-0.6cm}%表示图片与最上方的文字的距离
		\setlength{\abovecaptionskip}{0.28cm}% 表示caption 与图片之间的距离
		\includegraphics[width=1\linewidth]{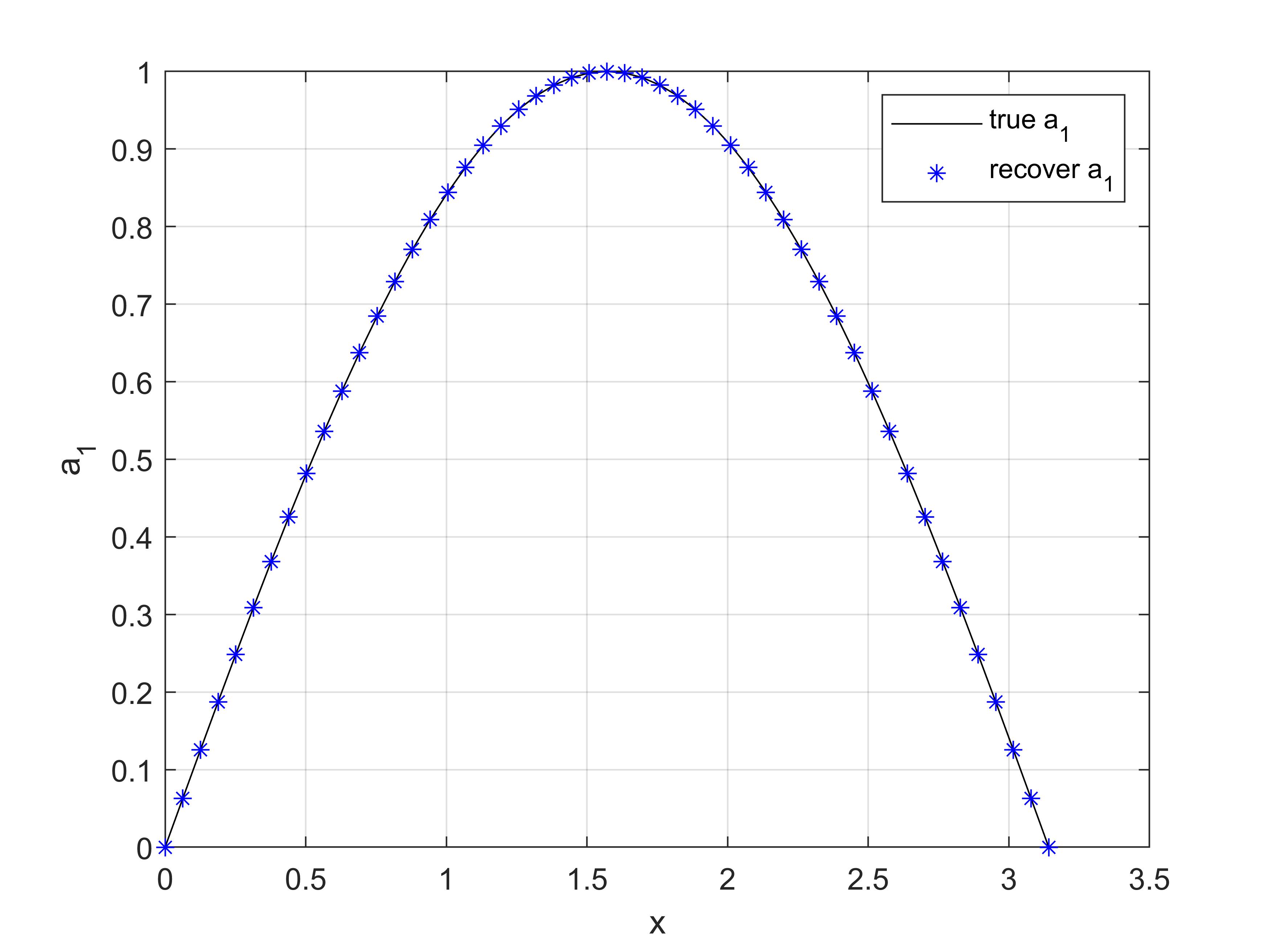}
		\caption{Recover $a_1(x)$ (POD).}
	  \end{minipage}    
      \caption{The case of noise-free data (Ex. \ref{ex2}).}
\end{figure}

\begin{figure}[!htb]
        \begin{minipage}{0.48\linewidth}% 表示图片的占用那一列的宽度
		\centering
		%\vspace{-0.6cm}%表示图片与最上方的文字的距离
		\setlength{\abovecaptionskip}{0.28cm}% 表示caption 与图片之间的距离
		\includegraphics[width=1\linewidth]{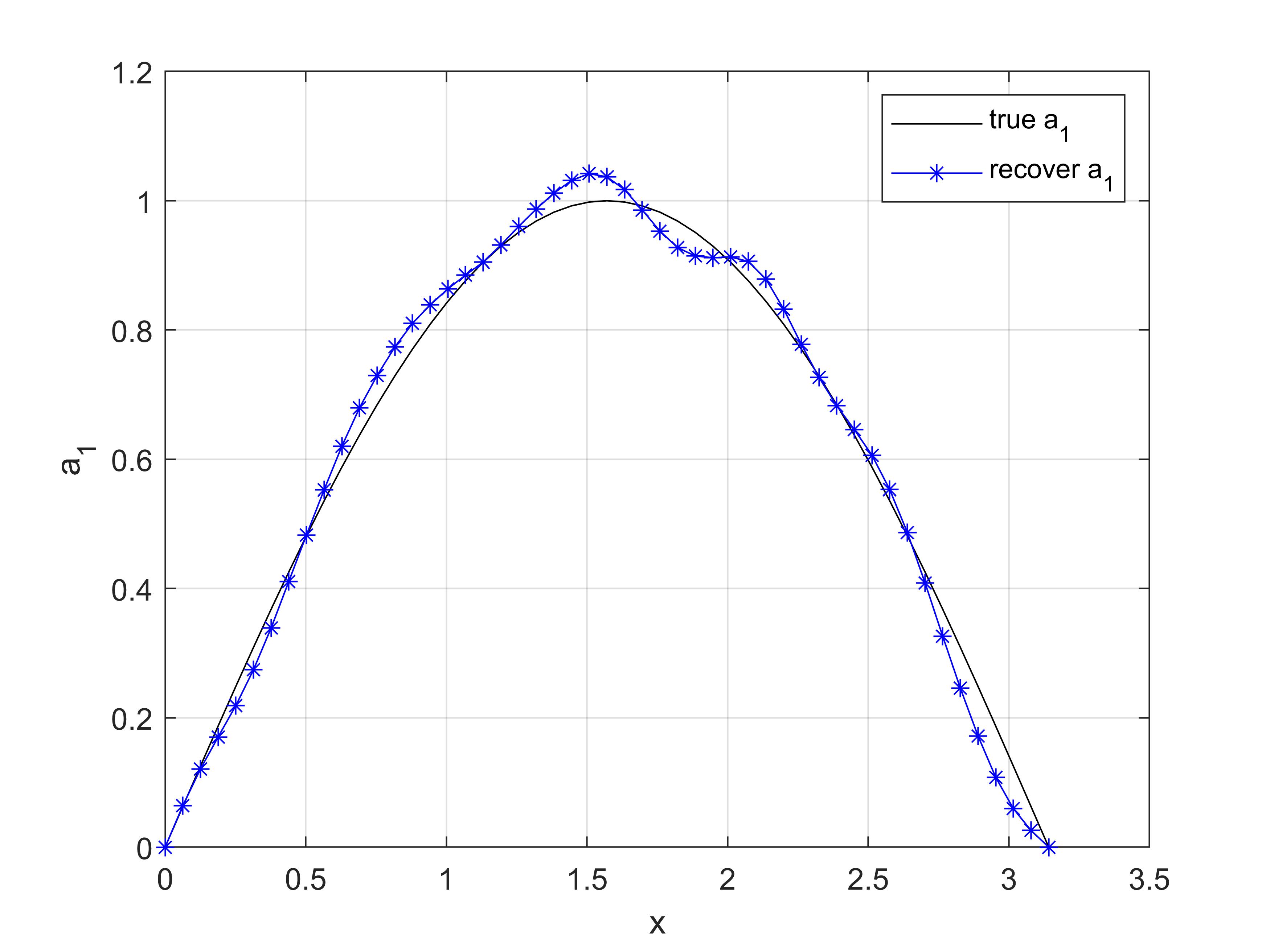}
		\caption{Recover $a_1(x)$ \cite{LiZhang2025}.}
	\end{minipage}
        \begin{minipage}{0.48\linewidth}% 表示图片的占用那一列的宽度
		\centering
		%\vspace{-0.6cm}%表示图片与最上方的文字的距离
		\setlength{\abovecaptionskip}{0.28cm}% 表示caption 与图片之间的距离
		\includegraphics[width=1\linewidth]{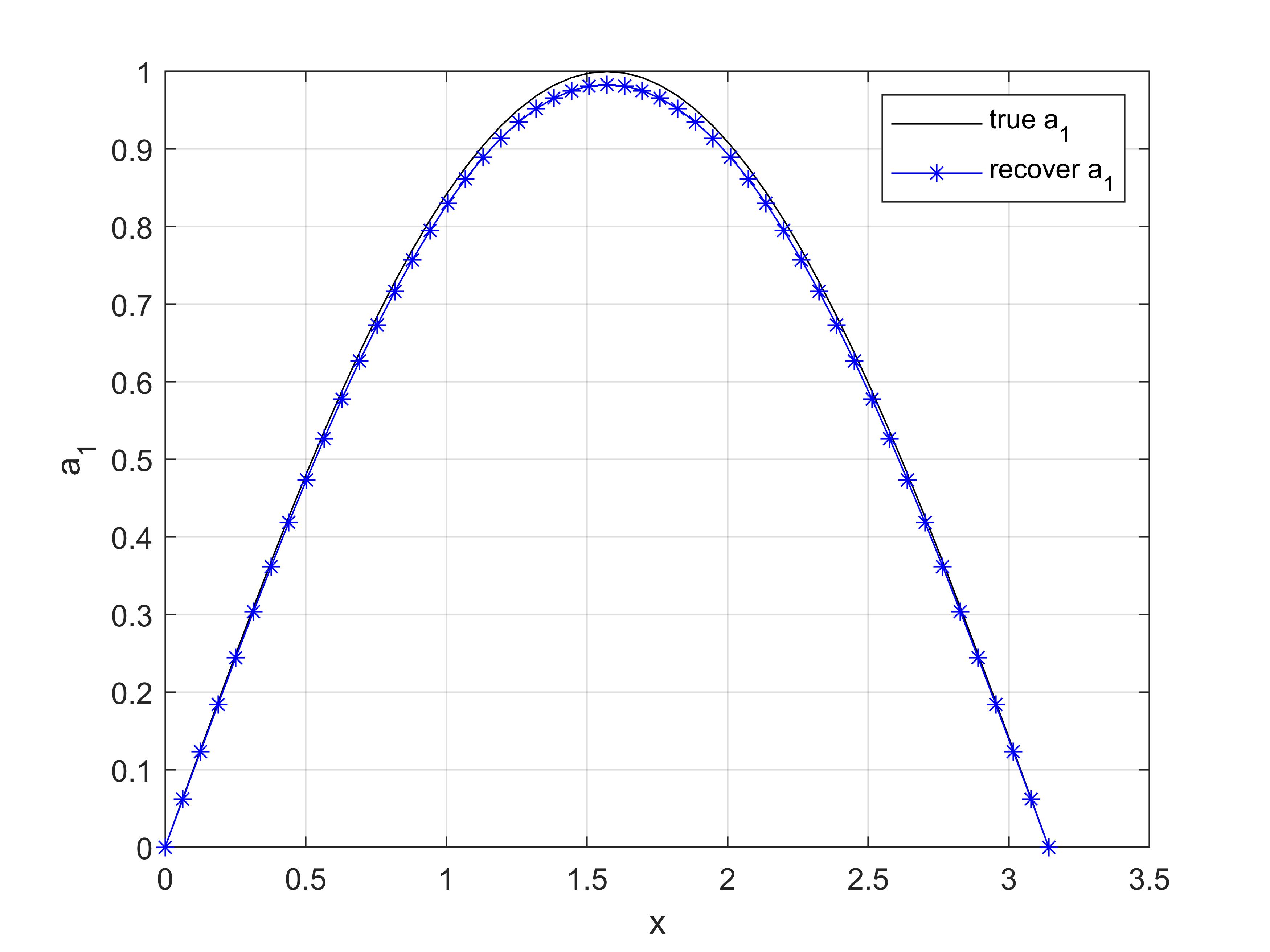}
		\caption{Recover $a_1(x)$ (POD).}
	  \end{minipage}    
     \caption{The case of noise data (Ex. \ref{ex2}).}
\end{figure}

\begin{example}\label{ex3}
We consider recovering more general initial functions. Substituting $a_1(x)$ in Examples \ref{ex1} and \ref{ex2} with the following case:
$$a_1=\begin{cases}1, &0< x\leq \frac{\pi}{2}\\
0, &\frac{\pi}{2} < x\leq \pi\end{cases}.$$
\end{example}

For case Example \ref{ex3}, the time costs are $491.48$ seconds for FEM \cite{LiZhang2025} and $46.96$ seconds for POD.
The noise-free observation data are obtained by numerical methods that contain the calculation error. The reconstruction process is very sensitive to errors, especially for discontinuous $a_1$. Hence, it is reasonable that the numerical results are not that good.

Then, we consider the case of noise data for Example \ref{ex3}. The noise level $\epsilon=\frac{\sigma}{\|u(T)\|_\infty}\approx \frac{0.015}{0.1}=15\%$. The regularization parameter $\lambda=1.26\times10^{-7}$. The time costs are $90.51$ seconds for FEM \cite{LiZhang2025} and $2.00$ seconds for POD. The POD basis functions of Ex. \ref{ex3} are presented. It illustrates how the solution can be represented in the basis function space.

\begin{figure}[!htb]
        \begin{minipage}{0.48\linewidth}% 表示图片的占用那一列的宽度
		\centering
		%\vspace{-0.6cm}%表示图片与最上方的文字的距离
		\setlength{\abovecaptionskip}{0.28cm}% 表示caption 与图片之间的距离
		\includegraphics[width=1\linewidth]{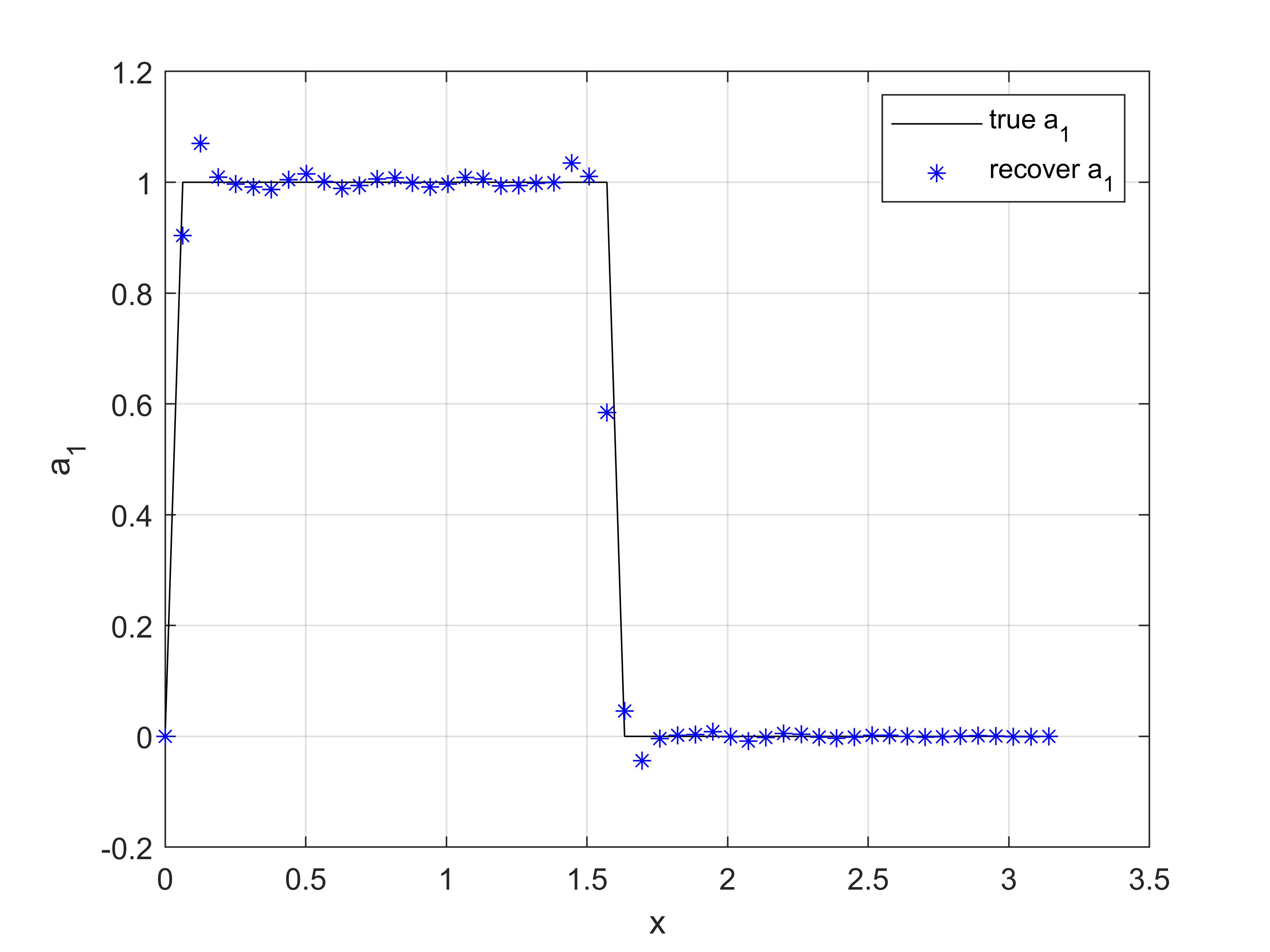}
		\caption{Recover $a_1(x)$ \cite{LiZhang2025}.}
	\end{minipage}
        \begin{minipage}{0.48\linewidth}% 表示图片的占用那一列的宽度
		\centering
		%\vspace{-0.6cm}%表示图片与最上方的文字的距离
		\setlength{\abovecaptionskip}{0.28cm}% 表示caption 与图片之间的距离
		\includegraphics[width=1\linewidth]{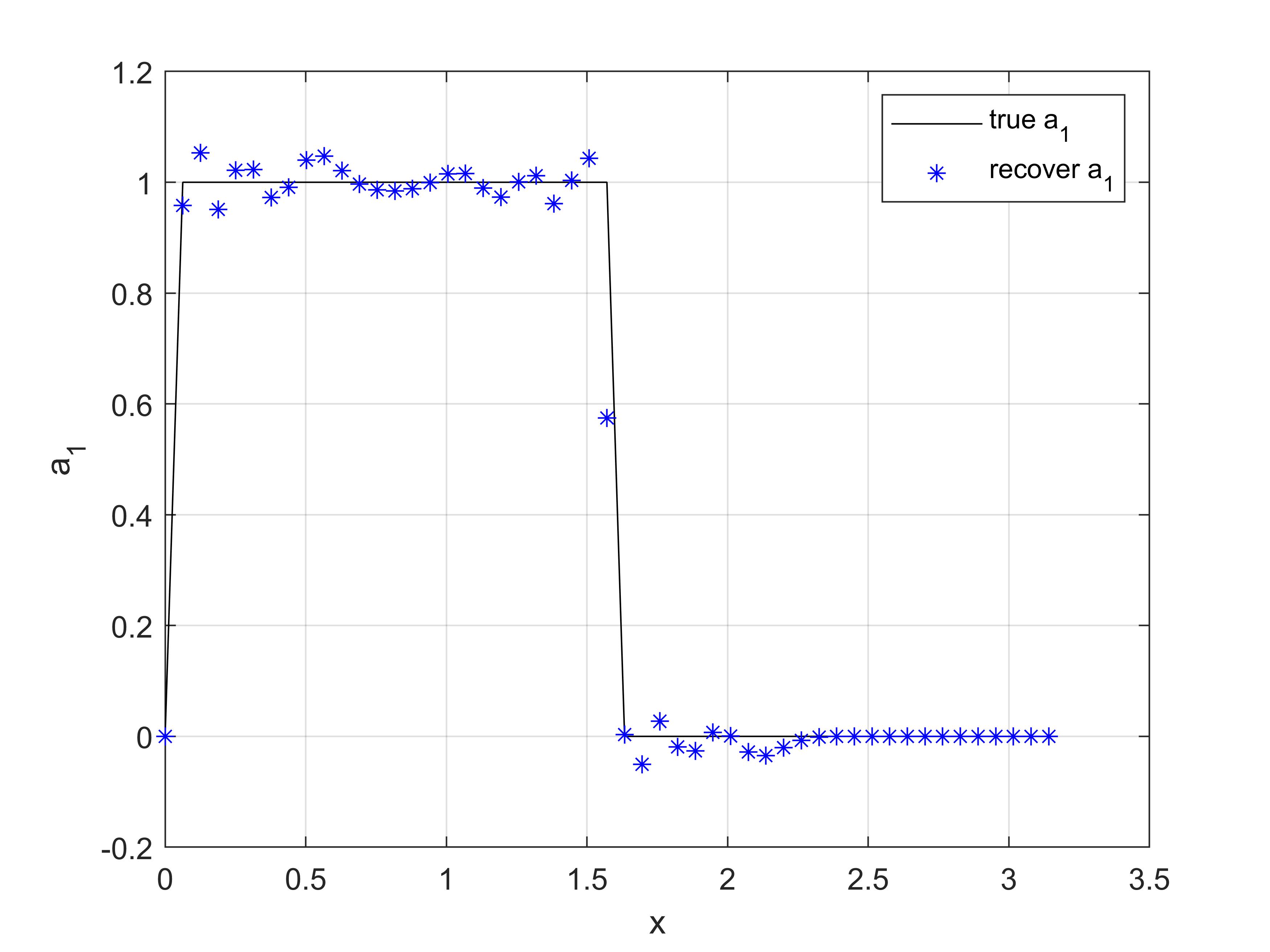}
		\caption{Recover $a_1(x)$ (POD).}
	  \end{minipage}    
      \caption{The case of noise-free data (Ex. \ref{ex3}).}
\end{figure}

\begin{figure}[!htb]
        \begin{minipage}{0.48\linewidth}% 表示图片的占用那一列的宽度
		\centering
		%\vspace{-0.6cm}%表示图片与最上方的文字的距离
		\setlength{\abovecaptionskip}{0.28cm}% 表示caption 与图片之间的距离
		\includegraphics[width=1\linewidth]{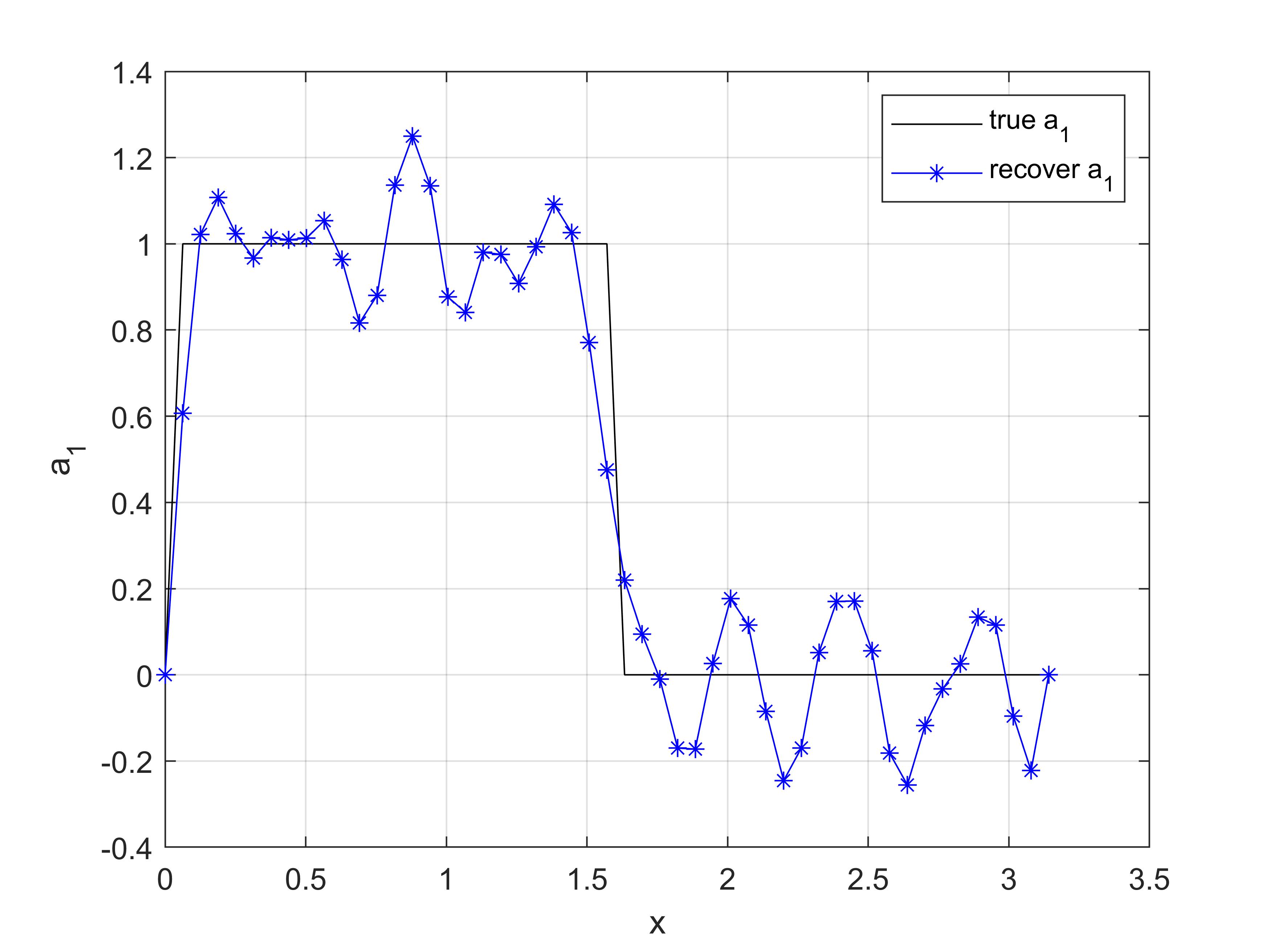}
		\caption{Recover $a_1(x)$ \cite{LiZhang2025}.}
	\end{minipage}
        \begin{minipage}{0.48\linewidth}% 表示图片的占用那一列的宽度
		\centering
		%\vspace{-0.6cm}%表示图片与最上方的文字的距离
		\setlength{\abovecaptionskip}{0.28cm}% 表示caption 与图片之间的距离
		\includegraphics[width=1\linewidth]{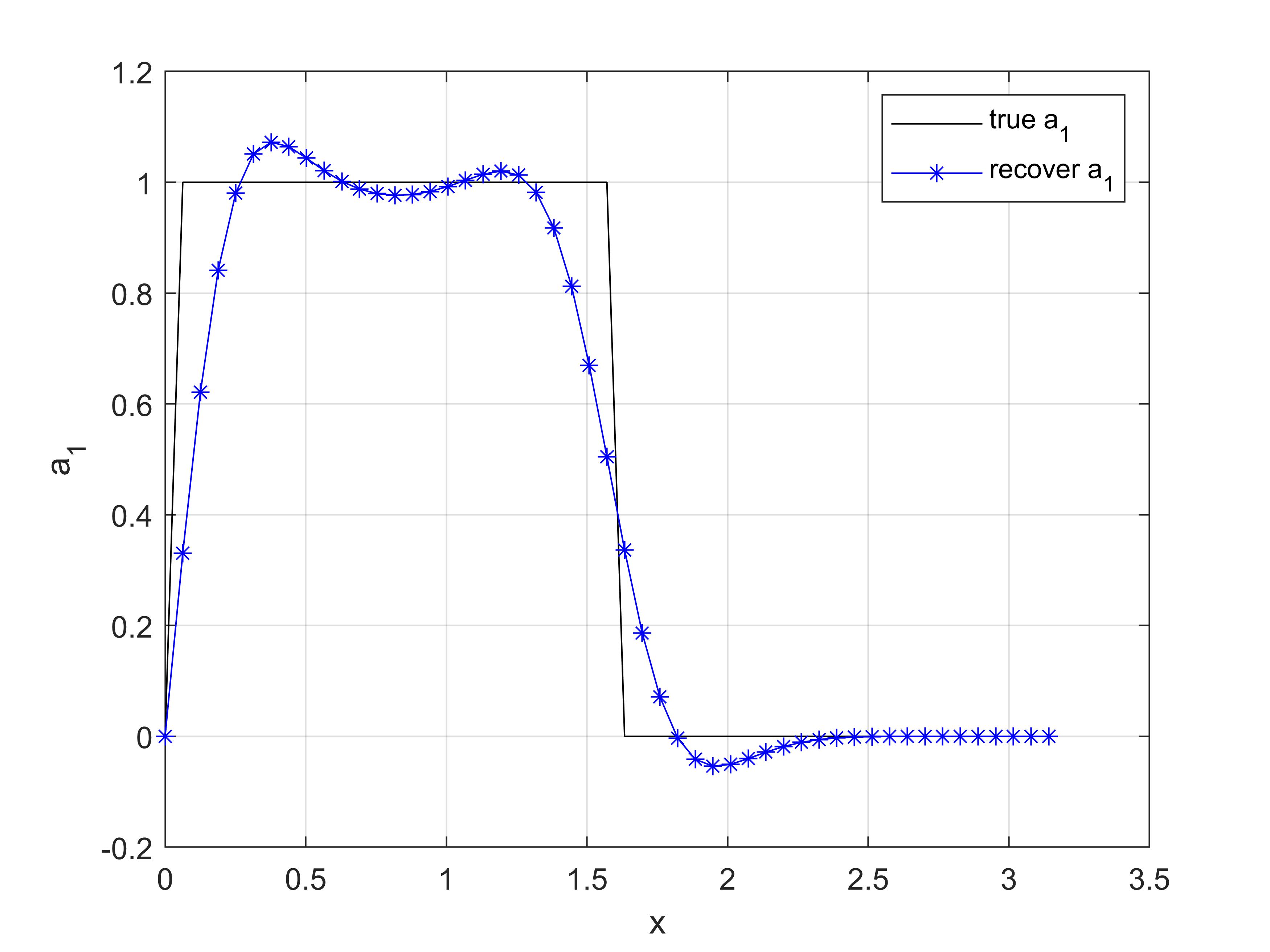}
		\caption{Recover $a_1(x)$ (POD).}
	  \end{minipage}    
    \caption{The case of noise data (Ex. \ref{ex3}).}
\end{figure}

\begin{figure}[!htb]
        \begin{minipage}{0.33\linewidth}% 表示图片的占用那一列的宽度
		\centering
		%\vspace{-0.6cm}%表示图片与最上方的文字的距离
		\setlength{\abovecaptionskip}{0.28cm}% 表示caption 与图片之间的距离
		\includegraphics[width=1\linewidth]{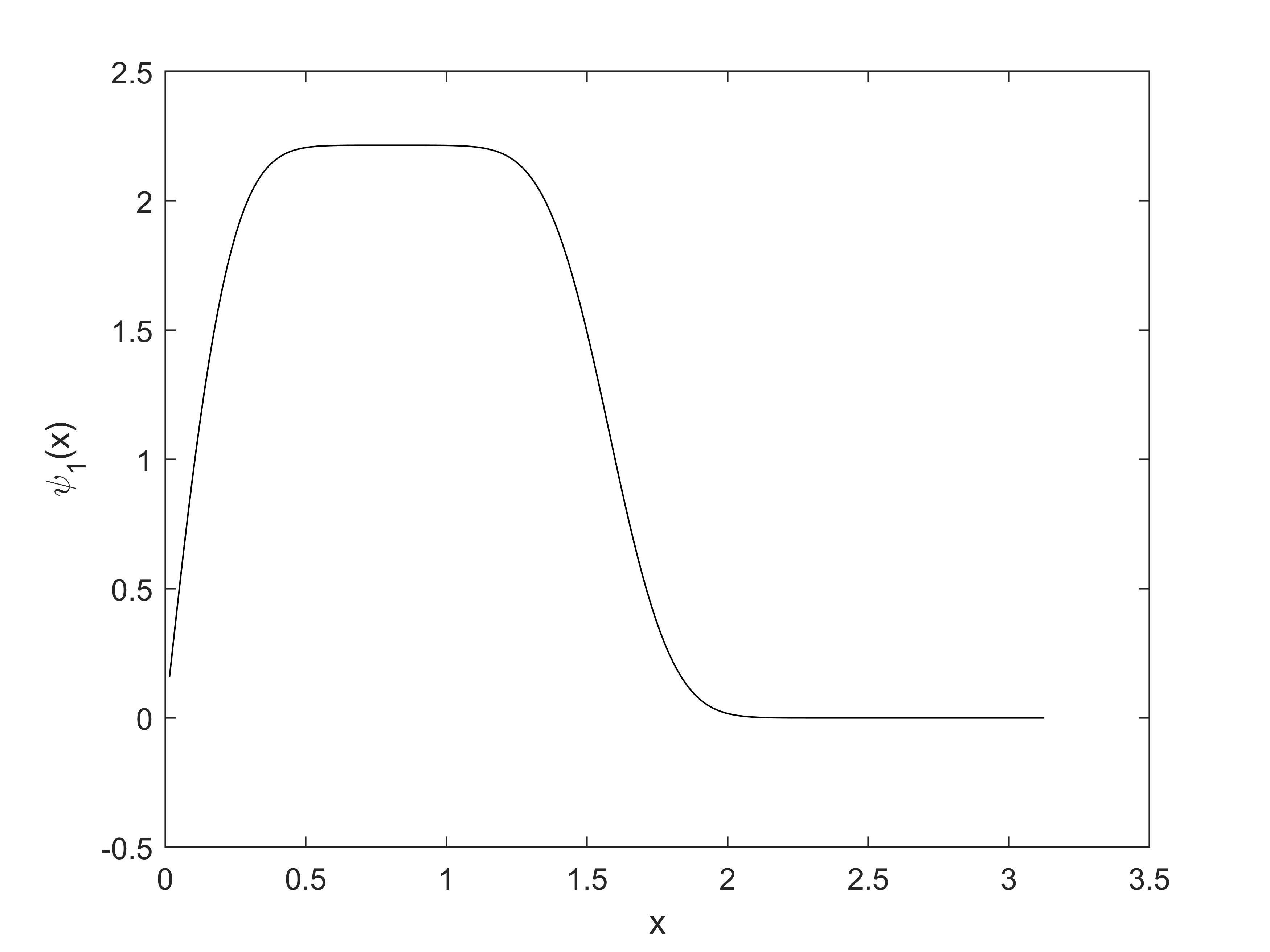}
		\caption{$\psi_1(x)$.}
	\end{minipage}
        \begin{minipage}{0.33\linewidth}% 表示图片的占用那一列的宽度
		\centering
		%\vspace{-0.6cm}%表示图片与最上方的文字的距离
		\setlength{\abovecaptionskip}{0.28cm}% 表示caption 与图片之间的距离
		\includegraphics[width=1\linewidth]{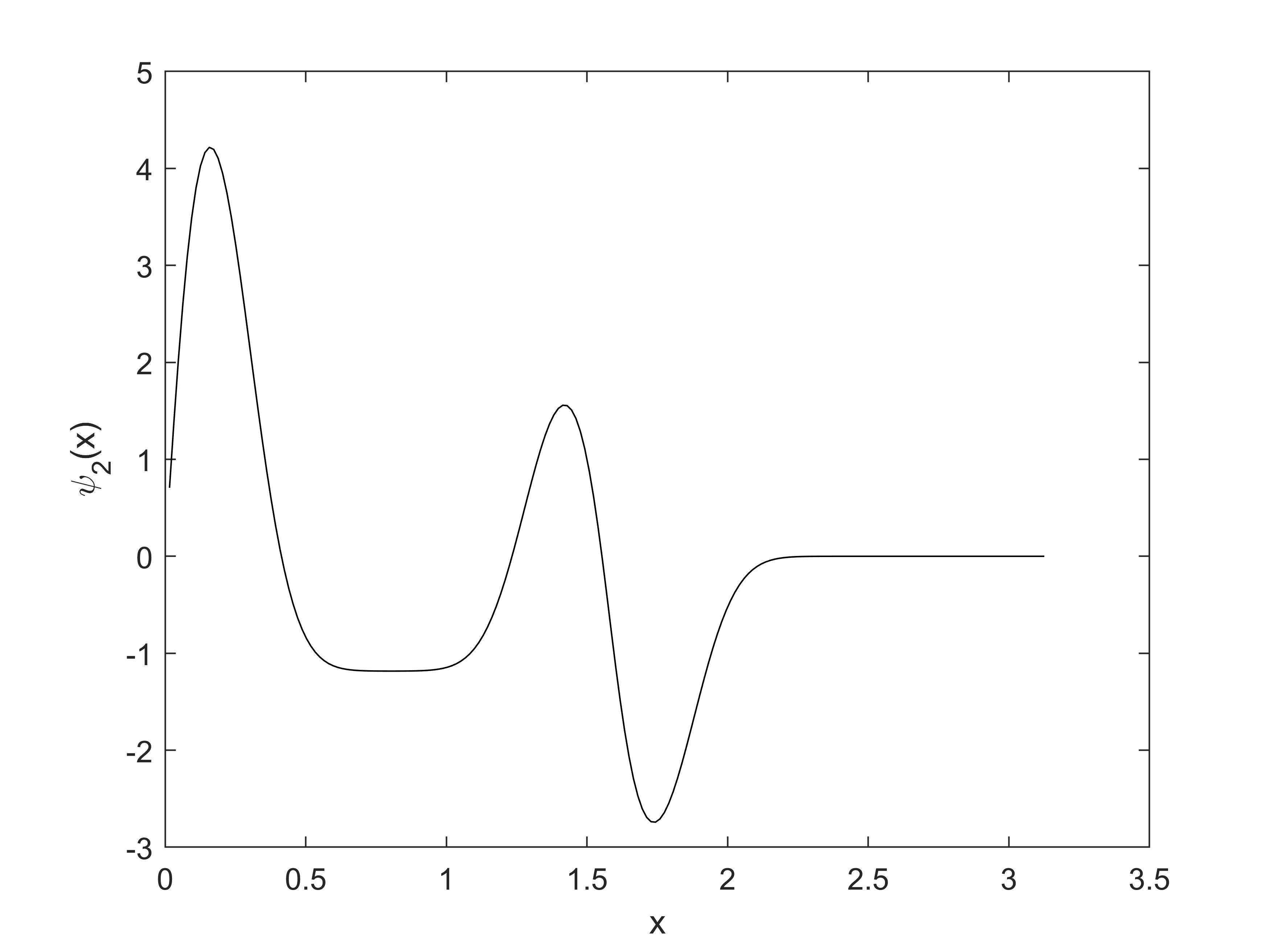}
		\caption{$\psi_2(x)$.}
	  \end{minipage} 
      \begin{minipage}{0.33\linewidth}% 表示图片的占用那一列的宽度
		\centering
		%\vspace{-0.6cm}%表示图片与最上方的文字的距离
		\setlength{\abovecaptionskip}{0.28cm}% 表示caption 与图片之间的距离
		\includegraphics[width=1\linewidth]{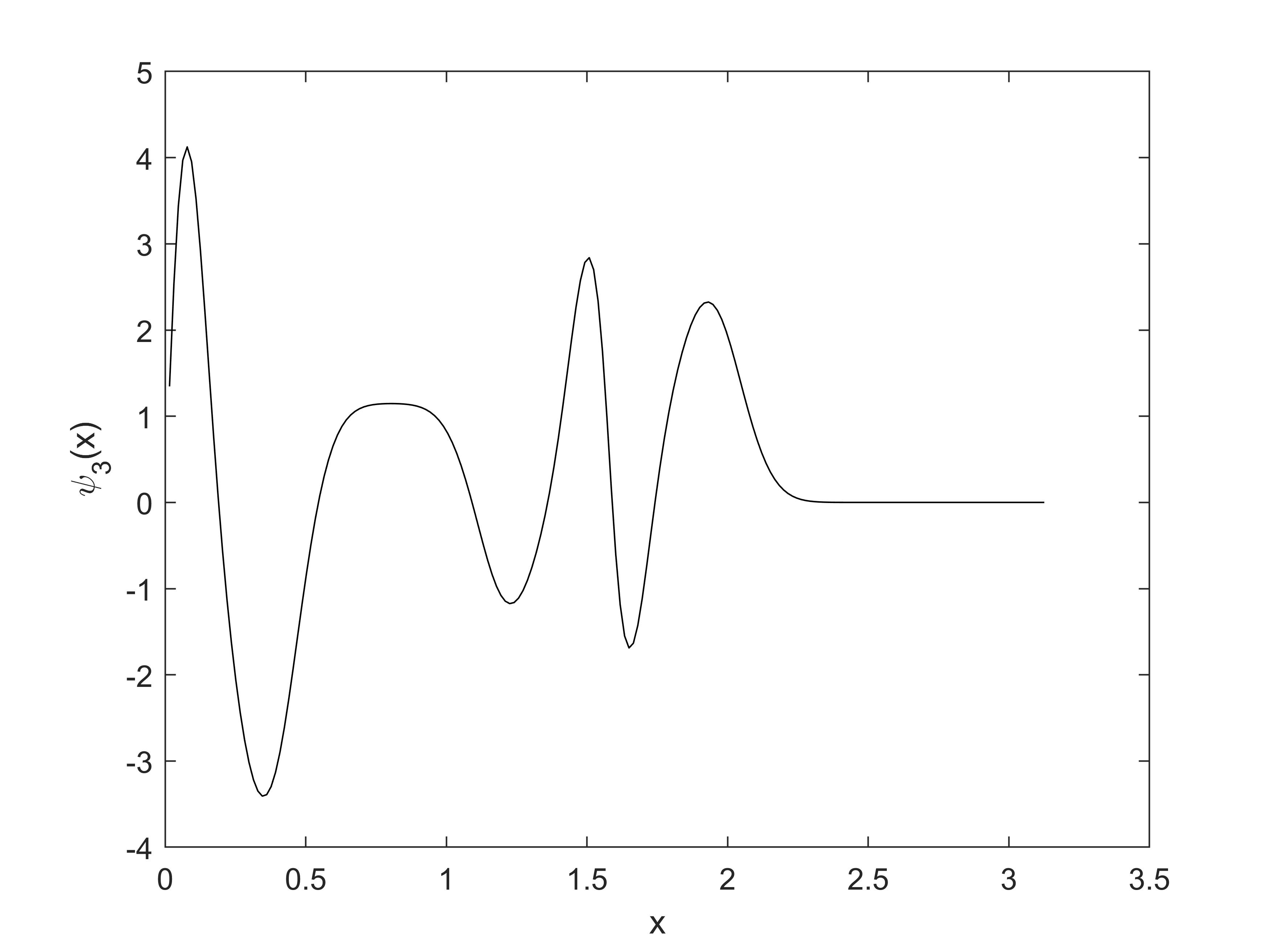}
		\caption{$\psi_3(x)$.}
	\end{minipage}\\
        \begin{minipage}{0.33\linewidth}% 表示图片的占用那一列的宽度
		\centering
		%\vspace{-0.6cm}%表示图片与最上方的文字的距离
		\setlength{\abovecaptionskip}{0.28cm}% 表示caption 与图片之间的距离
		\includegraphics[width=1\linewidth]{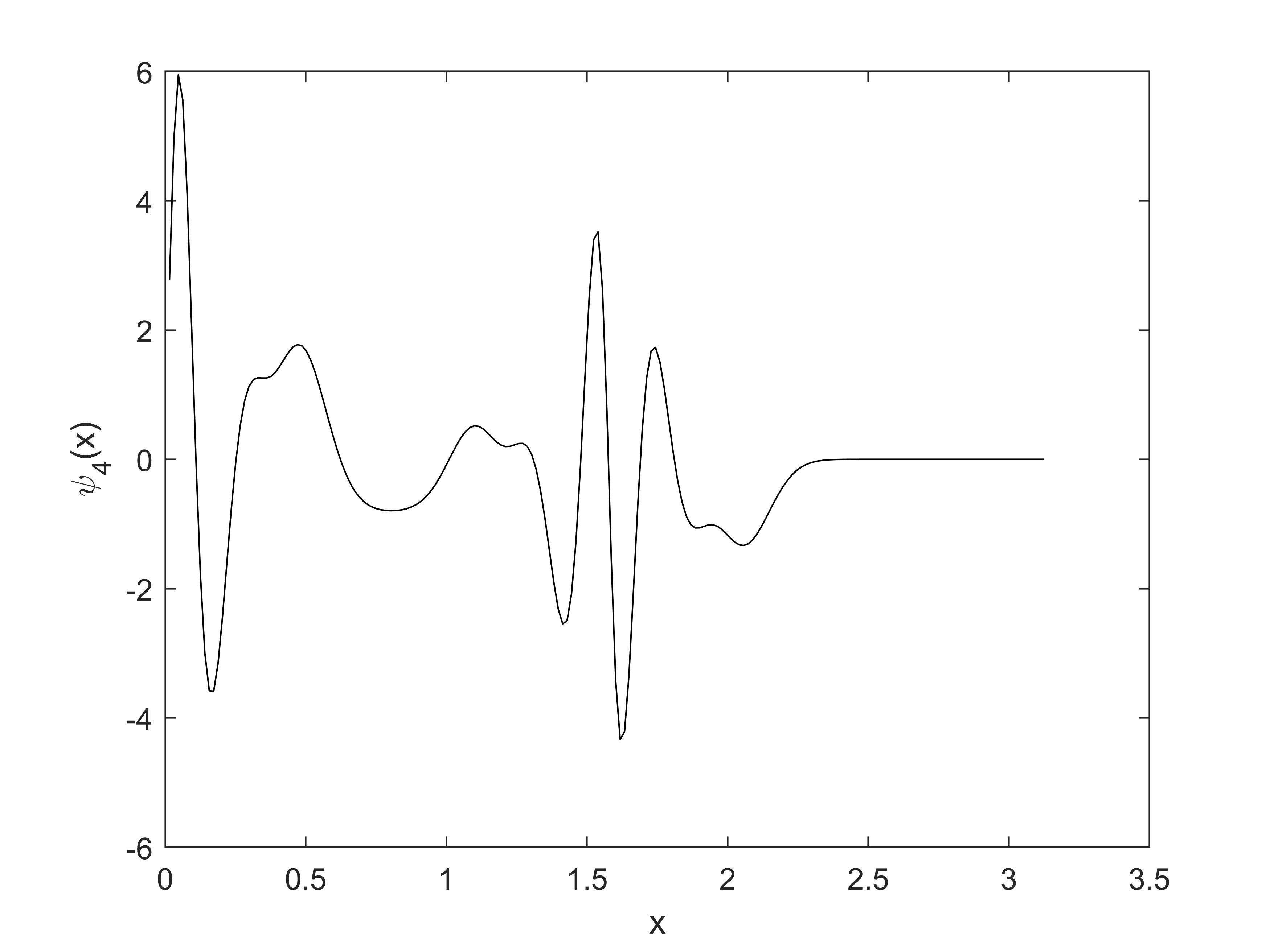}
		\caption{$\psi_4(x)$.}
	  \end{minipage} 
      \begin{minipage}{0.33\linewidth}% 表示图片的占用那一列的宽度
		\centering
		%\vspace{-0.6cm}%表示图片与最上方的文字的距离
		\setlength{\abovecaptionskip}{0.28cm}% 表示caption 与图片之间的距离
		\includegraphics[width=1\linewidth]{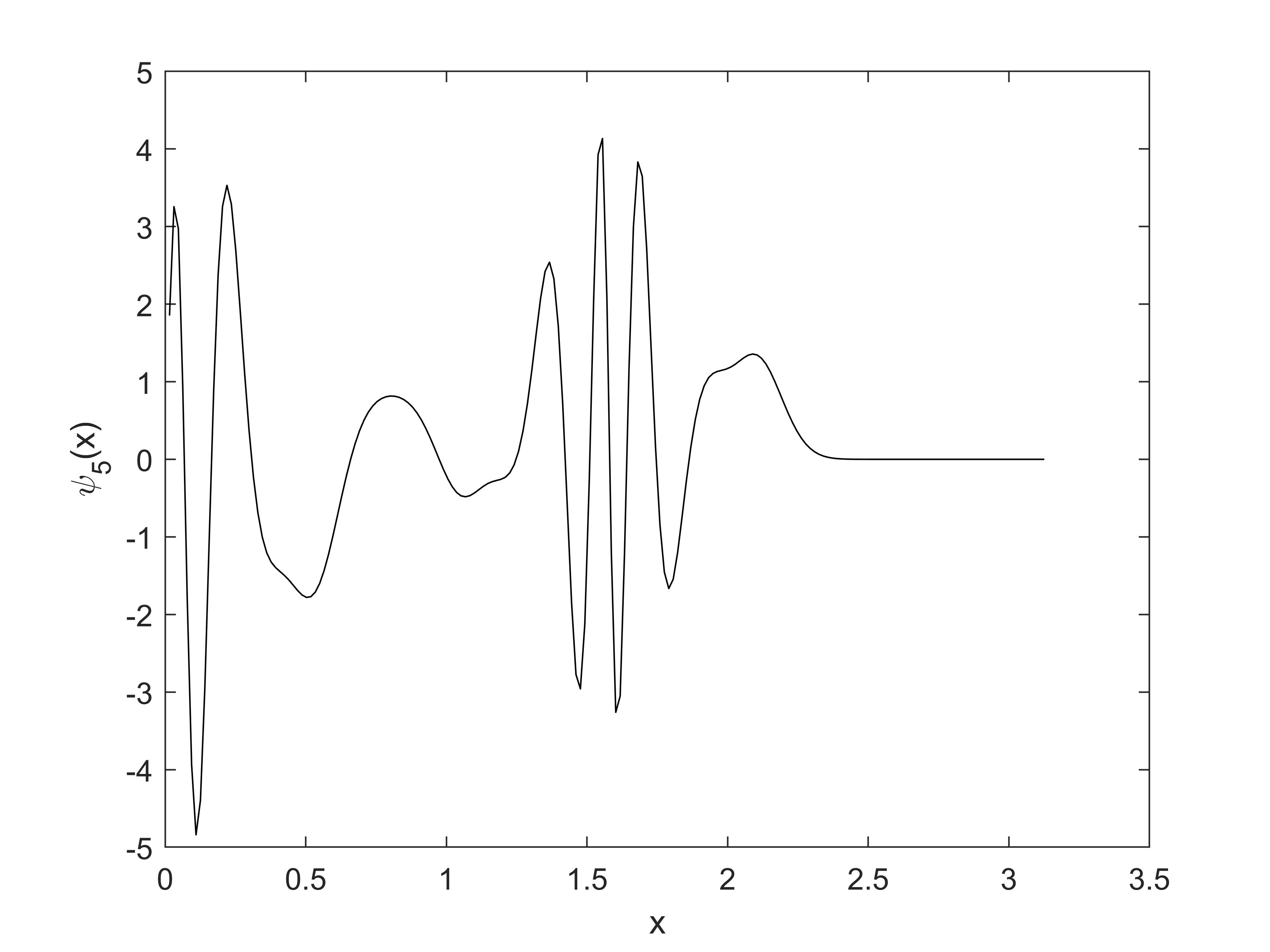}
		\caption{$\psi_5(x)$.}
	  \end{minipage} 
    \caption{The POD basis functions (Ex. \ref{ex3}).}
\end{figure}

\begin{example}\label{ex4}
We consider the following 2D cases in $\Omega:=[0,1]\times[0,1]$:
$$(a)~a_1=\sin(2\pi x)\sin(2\pi y),~(b)~a_1=x(1-x)\sin(2\pi y).$$
\end{example}
We take five POD basis functions in the calculation. The time grid $t_n=0.1(\frac{n}{N})^r$, $N=160$, $r=\frac{2-\nu}{1-\nu}$,  $\nu=\frac{5}{8}$, the space step $h=\frac{1}{30}$. For Example \ref{ex4} (a), the noise level $\epsilon=\frac{\sigma}{\|u(T)\|_\infty}\approx \frac{0.005}{0.023}=22\%$. The regularization parameter $\lambda=2.44\times10^{-7}$. The time costs are $258.23$ seconds for FEM \cite{LiZhang2025} and $1.06$ seconds for POD.

\begin{figure}[!htb]
        \begin{minipage}{0.48\linewidth}% 表示图片的占用那一列的宽度
		\centering
		%\vspace{-0.6cm}%表示图片与最上方的文字的距离
		\setlength{\abovecaptionskip}{0.28cm}% 表示caption 与图片之间的距离
		\includegraphics[width=1\linewidth]{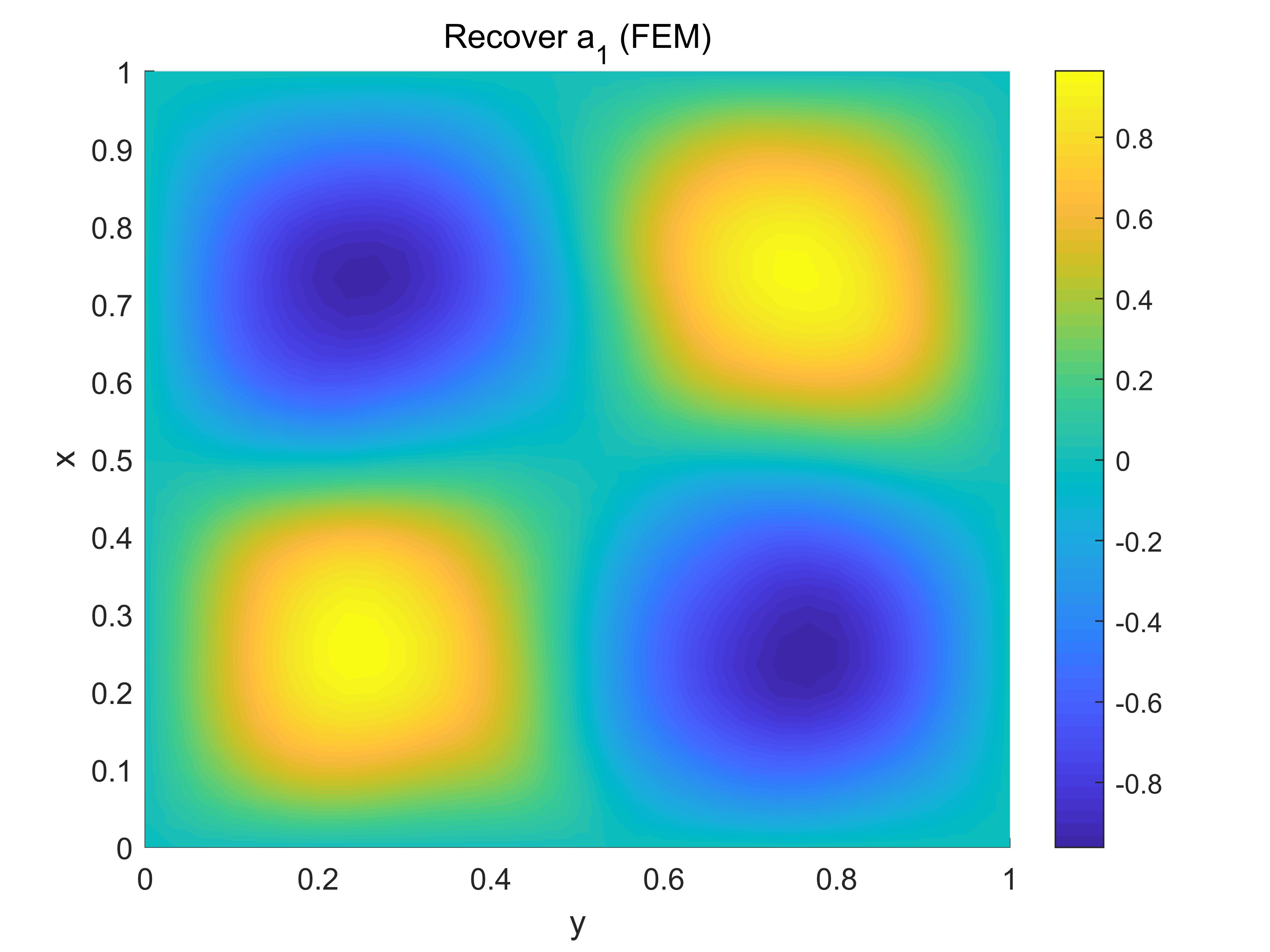}
		\caption{Recover $a_1(x)$ \cite{LiZhang2025}.}
	\end{minipage}
        \begin{minipage}{0.48\linewidth}% 表示图片的占用那一列的宽度
		\centering
		%\vspace{-0.6cm}%表示图片与最上方的文字的距离
		\setlength{\abovecaptionskip}{0.28cm}% 表示caption 与图片之间的距离
		\includegraphics[width=1\linewidth]{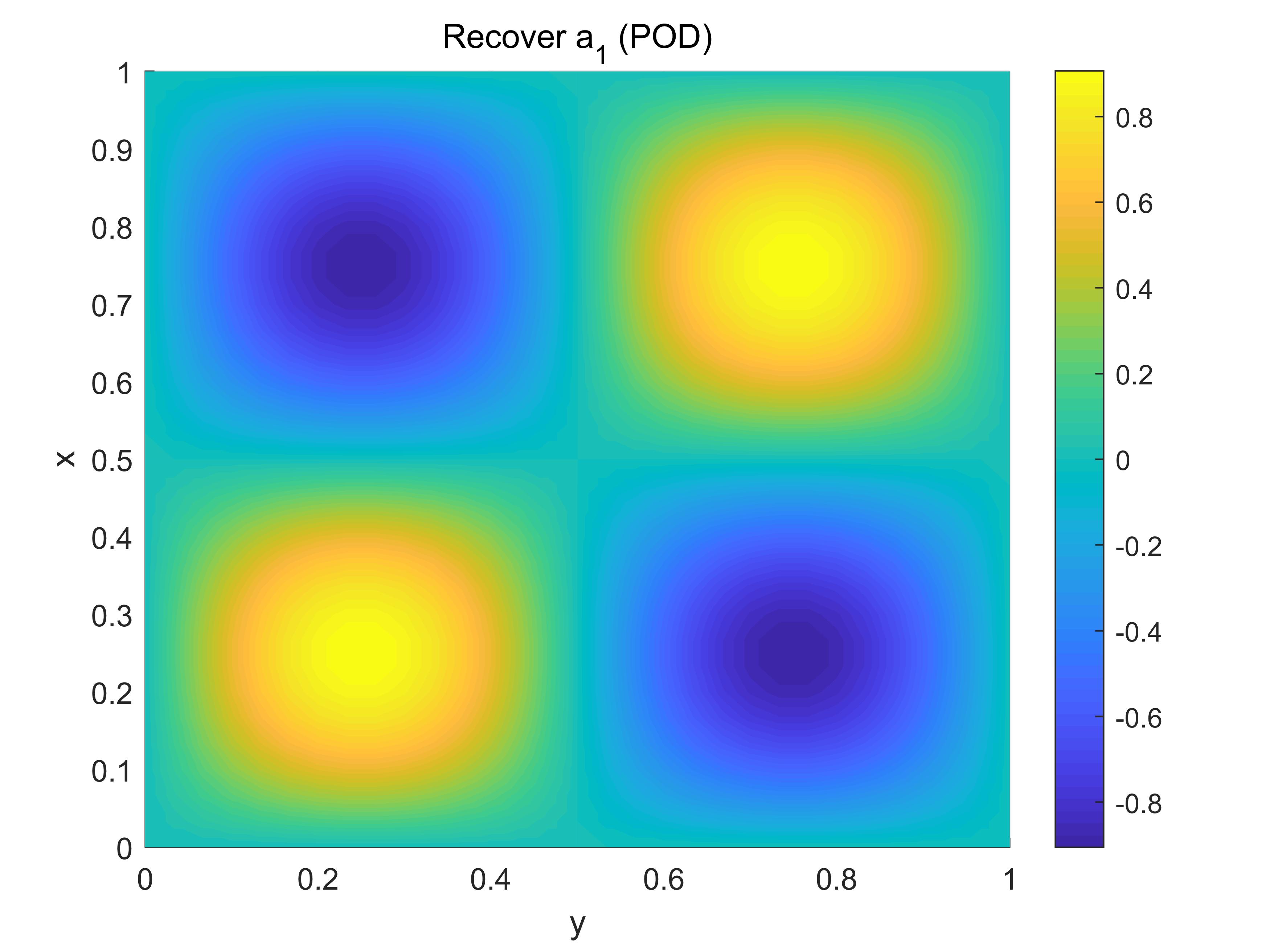}
		\caption{Recover $a_1(x)$ (POD).}
	  \end{minipage}    
     \caption{Numerical reconstruction Ex. \ref{ex4} (a).}
\end{figure}

For Example \ref{ex4} (b), the noise level $\epsilon=\frac{\sigma}{\|u(T)\|_\infty}\approx \frac{0.005}{0.0093}=54\%$. The regularization parameter $\lambda=2.6\times10^{-6}$. The time costs are $675.04$ seconds for FEM \cite{LiZhang2025} and $0.97$ seconds for POD. It shows that POD method also works well for 2D cases.

\begin{figure}[!htb]
        \begin{minipage}{0.48\linewidth}% 表示图片的占用那一列的宽度
		\centering
		%\vspace{-0.6cm}%表示图片与最上方的文字的距离
		\setlength{\abovecaptionskip}{0.28cm}% 表示caption 与图片之间的距离
		\includegraphics[width=1\linewidth]{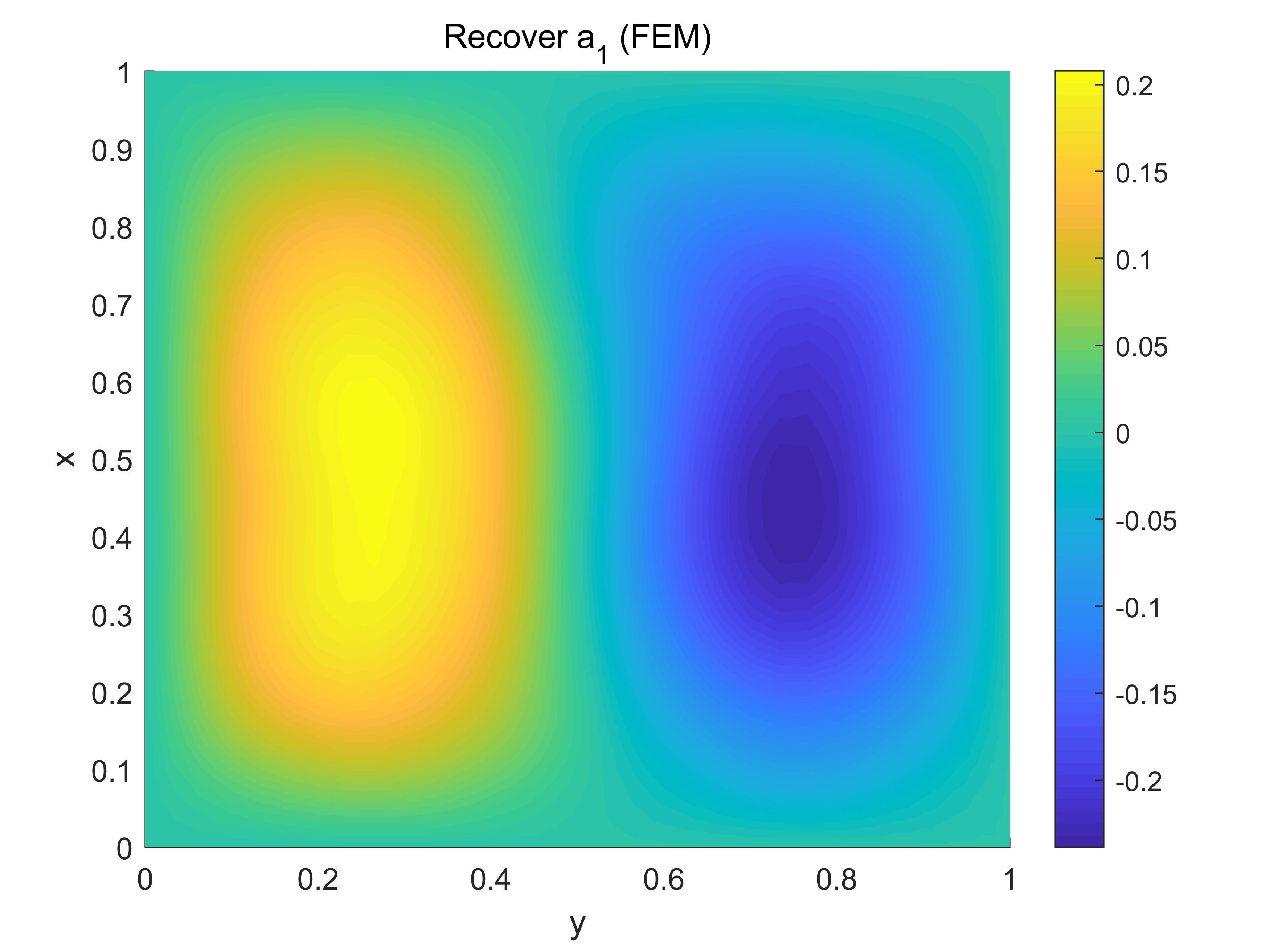}
		\caption{Recover $a_1(x)$ \cite{LiZhang2025}.}
	\end{minipage}
        \begin{minipage}{0.48\linewidth}% 表示图片的占用那一列的宽度
		\centering
		%\vspace{-0.6cm}%表示图片与最上方的文字的距离
		\setlength{\abovecaptionskip}{0.28cm}% 表示caption 与图片之间的距离
		\includegraphics[width=1\linewidth]{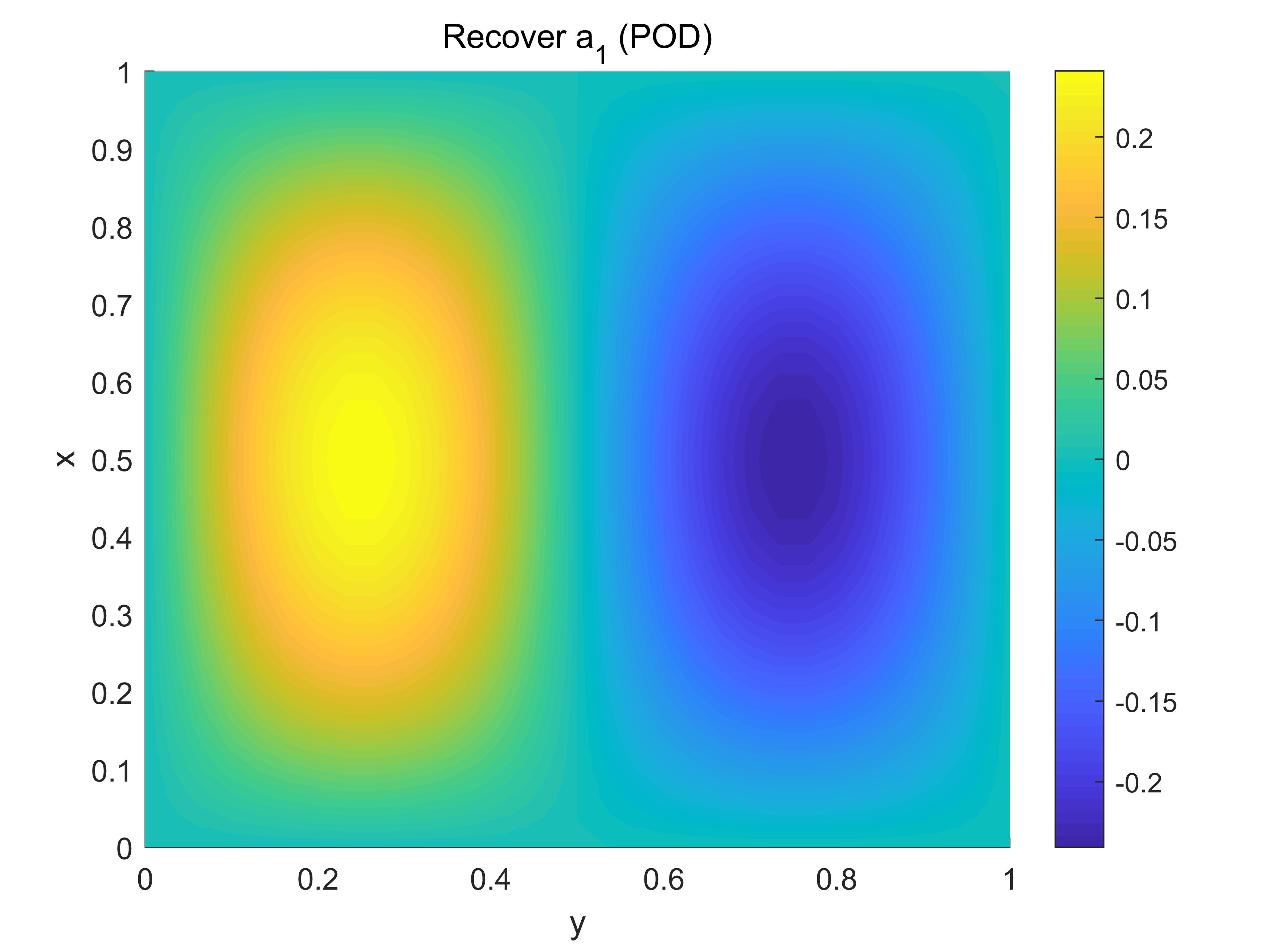}
		\caption{Recover $a_1(x)$ (POD).}
	  \end{minipage}    
     \caption{Numerical reconstruction Ex. \ref{ex4} (b).}
\end{figure}

\section{Concluding remarks}
A model reduction approach is proposed based on measurement data. It recovers the initial value in much less time and avoids the inverse crime. It provides a fast algorithm idea to improve the efficiency of solving the linear inverse problems based on the iteration optimization method.

\section*{Declarations}
On behalf of all authors, the corresponding author states that there is no conflict of interest. No datasets were generated or analyzed during the current study.

%\section*{Acknowledgments}
%Zhiyuan Li thanks the National Natural Science Foundation of China (no. 12271277) and Ningbo Youth Leading Talent Project (no. 2024QL045). The research of Wenlong Zhang is supported by the National Natural Science Foundation of China under grant numbers No.12371423 and No.12241104. This work is partially supported by the Open Research Fund of the Key Laboratory of Nonlinear Analysis \& Applications (Central China Normal University), Ministry of Education, China (no. NAA20230RG002).


\begin{thebibliography}{10}

\bibitem{LiZhang2025}
D.~Cen, Z.~Li, and W.~Zhang.
\newblock Numerical analysis of scattered point measurement-based
  regularization for backward problems for fractional wave equations.
\newblock {\em CSIAM Transactions on Applied Mathematics}, Accepted.

\bibitem{LiZhang2024}
D.~Cen, Z.~Li, and W.~Zhang.
\newblock Scattered point measurement-based regularization for backward
  problems for fractional wave equations.
\newblock {\em Journal of Scientific Computing}, 104:58, 2025.

\bibitem{CenZhang:2025}
Dakang Cen, Wenlong Zhang, and Junbin Zhong.
\newblock A randomized progressive iterative regularization method for data
  fitting problems.
\newblock {\em arXiv preprint arXiv:2506.03526}, 2025.

\bibitem{JinB2017sub_pod}
B.~Jin and Z.~Zhou.
\newblock An analysis of galerkin proper orthogonal decomposition for
  subdiffusion.
\newblock {\em ESAIM: Mathematical Modelling and Numerical Analysis},
  51:89--113, 2017.

\bibitem{LiaoH2018L1}
H.~Liao, D.~Li, and J.~Zhang.
\newblock Sharp error estimate of a nonuniform l1 formula for time-fractional
  reaction subdiffusion equations.
\newblock {\em SIAM Journal on Numerical Analysis}, 56:1112--1133, 2018.

\bibitem{LyuP2022SFOR}
P.~Lyu and S.~Vong.
\newblock A symmetric fractional-order reduction method for direct nonuniform
  approximations of semilinear diffusion-wave equations.
\newblock {\em Journal of Scientific Computing}, 93:34, 2022.

\bibitem{podlubny1999fractional}
I.~Podlubny.
\newblock {\em Fractional differential equations}, volume 198 of {\em
  Mathematics in Science and Engineering}.
\newblock Academic Press, Inc., San Diego, CA, 1999.
\newblock An introduction to fractional derivatives, fractional differential
  equations, to methods of their solution and some of their applications.

\bibitem{Sirovich1987}
L.~Sirovich.
\newblock Turbulence and the dynamics of coherent structures. i. coherent
  structures.
\newblock {\em Quarterly of applied mathematics}, 45(3):561--571, 1987.

\bibitem{utreras1988convergence}
Florencio~I. Utreras.
\newblock Convergence rates for multivariate smoothing spline functions.
\newblock {\em Journal of Approximation Theory}, 52(1):1--27, 1988.

\bibitem{Zhang2024pod}
Z.~Wang, W.~Zhang, and Z.~Zhang.
\newblock A data-driven model reduction method for parabolic inverse source
  problems and its convergence analysis.
\newblock {\em Journal of Computational Physics}, 487:112156, 2023.

\bibitem{ZhangZhang2024pod}
Z.~Zhang and Z.~Zhang.
\newblock A novel model reduction method for parabolic inverse problems without
  inverse crime.
\newblock {\em Journal of Scientific Computing}, 105:80, 2025.

\end{thebibliography}
\end{document}